\documentclass[a4paper,reqno,10pt]{amsart}
\usepackage{amsmath,amstext,amssymb,amsopn,amsthm,mathrsfs,mathabx}
\usepackage{xcolor}
\usepackage{tikz}
\usetikzlibrary{arrows}

\numberwithin{equation}{section}
\allowdisplaybreaks

\newtheorem{theorem}{Theorem}[section]

\newtheorem{lemma}[theorem]{Lemma}
\newtheorem{proposition}[theorem]{Proposition}

\theoremstyle{definition}

\newtheorem{remark}[theorem]{Remark}

\numberwithin{equation}{section}

\newcommand{\Z}{\mathbb{Z}}
\newcommand{\R}{\mathbb{R}}
\newcommand{\C}{\mathbb{C}}
\newcommand{\Rd}{\mathbb{R}^d}

\newcommand*\calC{\mathcal{C}}

\newcommand*\calF{\mathcal{F}}

\newcommand*\calN{\mathcal{N}}
\newcommand*\calR{\mathcal{R}}
\newcommand*\calL{\mathcal{L}}

\newcommand{\sgn}{{\rm det}}
\newcommand{\Hom}{\eta\in{\rm Hom}(W,\,\widehat{\mathbb Z}_2)}

\def\a{\alpha}
\def\b{\beta}

\def\ve{\varepsilon}

\begin{document}

\baselineskip=17pt

\title[Sharp estimates for heat kernels on Weyl chambers]{Genuinely sharp estimates for heat kernels on Weyl chambers}

\author[K. Stempak]{Krzysztof Stempak}
\address{55-093 Kie\l{}cz\'ow, Poland}        
\email{Krzysztof.Stempak@pwr.edu.pl}
\begin{abstract}  
We consider heat kernels on Weyl chambers corresponding to Laplacians subject to mixed Dirichlet-Neumann boundary conditions imposed on the boundary. 
Using purely analytic tools we prove genuinely sharp two-sided global estimates in the frameworks of orthogonal root systems and the dihedral root 
systems $I_2(3)$ and $I_2(4)$.
\end{abstract}



\maketitle

\section{Introduction and statement of results} \label{sec:in}
Heat kernel estimates attracted wide attention as useful tools of great importance in analysis. In the context of Euclidean domains, with the Dirichlet or 
Neumann Laplacians as self-adjoint operators, a variety of settings were considered, including bounded and unbounded domains, domains with smoothness 
regularity conditions (of different degrees) imposed on their boundaries, etc. 

To give a flavor of such estimates we mention that for any convex domain $U\subset \Rd$, the Neumann heat kernel $p_U^N(t,x,y)$ admits qualitatively 
sharp global upper and lower \textit{Gaussian bounds} of the type  
\begin{equation}\label{convex}
p_U^N(t,x,y)\simeq\simeq \frac 1{|B(x,\sqrt t)\cap U|}\exp(-c\|x-y\|^2/t), \qquad x,y\in U,\quad t>0,
\end{equation}
(see e.g. \cite[p.\,4]{GSC}), where $B(x,r)$ stands for the Euclidean ball of center $x$ and radius $r$; see the end of this section, where the meaning 
of the symbols ``$\simeq\simeq$'' and ``$\simeq$'' is explained. Similar estimates in more general \textsl{inner uniform domains} are also known; 
see \cite[Section 3]{GSC} or \cite[Theorem 5.4]{SC}. 

For the Dirichlet heat kernel $p_U^D(t,x,y)$ it was proved by Gyrya and Saloff-Coste, see \cite[Section 5]{GSC} or \cite[Theorem 5.7]{SC}, that for an 
unbounded \textit{inner uniform domain} $U$ (cf. \cite[Section 3]{GSC} for the definition) it holds
\begin{equation}\label{convex2}
p_U^D(t,x,y)\simeq\simeq \frac{h(x)h(y)}{\sqrt{V_{h^2}(x,\sqrt t)V_{h^2}(y,\sqrt t)}}\exp(-c d_U(x,y)^2/t), \qquad x,y\in U,\quad t>0;
\end{equation}
here $d_U(x,y)$ stands for the \textit{inner distance} in $U$ between $x$ and $y$ obtained by minimizing the length of curves joining $x$ and $y$ in $U$  
(thus, for convex $U$, $d_U(x,y)=\|x-y\|$), $h$ is the \textit{harmonic profile} of $U$ (called also the \textit{r\'eduite} of $U$, see 
\cite[Definition 5.5]{SC} or \cite[Section 4.1]{GSC}), and $V_{h^2}(x,r)=\int_{\{y\in U\colon d_U(x,y)<r\}} h^2$.   

The aim of this paper is to discuss global upper and lower bounds of heat kernels on Weyl chambers. For a chamber $C_+$, related to a root system $R$ in 
$\Rd$ with $W=W(R)$ as the corresponding Weyl group, these kernels, denoted $\{p_t^{\eta,C_+}\}_{t>0}$, are associated to nonnegative self-adjoint 
realizations $-\Delta^+_\eta$ in $L^2(C_+)$ of the Laplacian on $C_+$, subject to mixed Dirichlet-Neumann boundary conditions imposed on the 
facets of $C_+$. These conditions are determined by a homomorphism $\eta\in{\rm Hom}(W,\,\widehat{\mathbb Z}_2)$, where $\widehat{\mathbb Z}_2=\{1,-1\}$ 
with multiplication, and $\eta=$\,\textsl{det}  and $\eta\equiv1$ (the latter denoted in the sequel by \textsl{triv}) correspond to the Dirichlet and Neumann boundary conditions, respectively. As the main result we obtain genuinely sharp estimates of heat kernels in the framework of orthogonal 
root systems in $\Rd$, and in the context of the dihedral root systems $I_2(3)$ and $I_2(4)$  in $\R^2$.

Let 
$$
p_t^{(d)}(w)=(4\pi t)^{-d/2}\exp\big(-\|w\|^2/4t\big), \qquad w\in\Rd,
$$
be the $d$-dimensional Gauss-Weierstrass kernel. 

In the case of the orthogonal root system $R_k:=\{\pm e_{d-k+1},\ldots,\pm e_d\}$, $1\le k\le d$, and $\eta\in{\rm Hom}(W(R_k),\,\widehat{\mathbb Z}_2)\simeq  
\mathbb Z_2^k$ (group isomorphism), with $\R_{+,k}=\mathbb{R}^{d-k}\times(0,\infty)^k$ as the corresponding Weyl chamber, and with notation 
$J_\eta=\{1\le j \le k\colon \eta(j)=1\}$ (so that $J_{\textbf{0}}=\emptyset$ and $J_{\textbf{1}}=\{1,\ldots,k\}$, where $\textbf{0}$ stands for \textsl{triv} 
and $\textbf{1}$ stands for \textsl{det} in this setting), our first result reads as follows. 
\begin{proposition} \label{prop:1}
Let $1\le k\le d$ and $\eta\in \mathbb Z_2^k$. Then, uniformly in $t>0$ and $x,y\in \R_{+,k}$,
\begin{equation}\label{orto}
p_t^{\eta,\R_{+,k}}(x,y)\simeq  p_t^{(d)}(x-y)\prod_{j\in d-k+ J_\eta}\frac{x_{j}y_{j}}{x_{j}y_{j}+t}.
\end{equation}
\end{proposition}
The proof of Proposition \ref{prop:1} is furnished in Section \ref{sec:ort}.

In the case of $I_2(3)$ the associated Weyl chamber is 
$$
C_+^{(3)}=\big\{\rho e^{i \theta}\colon \rho>0,\,\,\,\,|\theta|<\pi/6\}=\{(x_1,x_2)\colon |x_2|<\frac1{\sqrt3}x_1\big\},
$$
and for the corresponding Dirichlet kernel, now denoted $\{p_t^{\sgn,C_+^{(3)}}\}_{t>0}$, our first main result reads as follows.
\begin{theorem} \label{thm:1}
We have, uniformly in $x=(x_1,x_2)\in C_+^{(3)},\,y=(y_1,y_2)\in C_+^{(3)}$, and  $t>0$, 
\begin{align}\label{sgn1}
p_t^{\sgn,C_+^{(3)}}(x,y)\simeq &\frac{x_1y_1}{x_1y_1+t}\, \frac{(x_1-\sqrt3 x_2)(y_1-\sqrt3 y_2)}{(x_1-\sqrt3 x_2)(y_1-\sqrt3 y_2)+t}
\, \frac{(x_1+\sqrt3 x_2)(y_1+\sqrt3 y_2)}{(x_1+\sqrt3 x_2)(y_1+\sqrt3 y_2)+t}\times \nonumber\\
&p_t^{(2)}(x-y).
\end{align}
\end{theorem}

In the case of $I_2(4)$ the associated Weyl chamber is 
$$
C_+^{(4)}=\big\{\rho e^{i \theta}\colon \rho>0,\,\,\,\,0<\theta<\pi/4\}=\{(x_1,x_2)\colon 0<x_2<x_1\},
$$
and for the corresponding Dirichlet kernel $\{p_t^{\sgn,C_+^{(4)}}\}_{t>0}$, our second main result is as follows.
\begin{theorem} \label{thm:2}
We have, uniformly in $x=(x_1,x_2)\in C_+^{(4)},\,y=(y_1,y_2)\in C_+^{(4)}$, and  $t>0$, 
\begin{equation}\label{sgn}
p_t^{\sgn,C_+^{(4)}}(x,y)\simeq \frac{x_1y_1}{x_1y_1+t}\,\frac{x_2y_2}{x_2y_2+t}\,\frac{(x_1-x_2)(y_1-y_2)}{(x_1-x_2)(y_1-y_2)+t}
\frac{(x_1+x_2)(y_1+y_2)}{(x_1+x_2)(y_1+y_2)+t}\,p_t^{(2)}(x-y).
\end{equation}
\end{theorem}
Since the last factor in the expression preceding $p_t^{(2)}(x-y)$ in \eqref{sgn} is comparable with the first factor, 
the right-hand side of  \eqref{sgn} can be reformulated to a slightly different form; see the comment following  Conjecture in
Section \ref{sec:app} concerning more general context to which the above case applies. We also remark that  the right-hand sides  of \eqref{sgn1} 
and \eqref{sgn} correctly reflect the behavior of $p_t^{{\rm det},C_+^{(3)}}(x,y)$ and $p_t^{{\rm det},C_+^{(4)}}(x,y)$ on the boundaries of 
$C_+^{(3)}$ and $C_+^{(4)}$, respectively. See Section \ref{sec:prel}, where the behavior of the $\eta$-heat kernels on the boundary 
of Weyl chambers is commented.

Apart from $\eta=\sgn$ in the case of Theorem \ref{thm:2}, similar results for heat kernels corresponding to two other nontrivial homomorphisms $\eta$ 
in the setting of $W(I_2(4))$ are contained in Propositions \ref{pro:N2} and \ref{pro:N1}. An example of a result concerning heat kernel estimates for 
the Laplacian subject to mixed Neuman-Dirichlet boundary conditions can be found, for instance, in  \cite[Section 1.7]{GSC}. 

Bounds of heat kernels in various settings were extensively investigated using both, probabilistic and analytic methods; see, for instance, \cite{Cou}, 
\cite{GSC}, \cite{SC}, \cite{V}, \cite{Z}, and references therein. 
It was noted in Nowak,  Sj\"ogren, and  Szarek \cite{NSSz} (see also \cite{NSSz0}) that ``\textit{Compared with qualitatively sharp estimates, 
genuinely sharp heat kernel bounds are in general harder to prove and appear rarely in the literature}.'' See, for instance, Ma\l{}ecki and Serafin \cite{MSer}, where genuinely sharp estimates in the setting of the Dirichlet Laplacian on the Euclidean ball were established (see also the references therein for other settings). In Nowak,  Sj\"ogren, and  Szarek \cite{NSSz0} such estimates were proved for the spherical heat kernel and in \cite{NSSz} for heat kernels on all compact rank-one symmetric spaces (and also for Jacobi expansions, and in the context of a ball and a simplex). 
See also the recent paper by Serafin \cite{Ser}, where estimates precisely describing the exponential behavior of Dirichlet heat kernels in convex domains with 
$\calC^{1,1}$ boundary were established (clearly, Weyl chambers are not such domains, except, say, half-spaces).

In the settings considered in this paper we use the formula contained in \eqref{first} for $p_t^{\eta,C_+}$. For $\eta\neq$\,\textsl{triv} it is given by an alternating sum. For $\eta=$\,\textsl{triv}, \eqref{first} is non-alternating and hence the estimates of the Neumann heat kernel are simple. Namely, we have the genuinely sharp estimate
\begin{equation}\label{Neu}
p_t^{{\rm triv},\,C_+}(x,y)\simeq p_t^{(d)}(x-y), \qquad x,y\in C_+,\quad t>0.
\end{equation}
This is because in the non-alternating sum given by \eqref{first} (when $\eta=$\,\textsl{triv}), the term corresponding to $g=e$ is dominant (for every 
$e\neq g\in W$, and $x,y\in C_+$, one has $\|gx-y\|>\|x-y\|$). Therefore, in what follows we focus on the case $\eta\neq$\,\textsl{triv}. For the  
domain $C_+$, which is convex, we have $|B(x,r)\cap C_+|\simeq r^d$, uniformly in $x\in C_+$ and $r>0$; this easily follows by the geometry of the situation. Hence in the considered case, \eqref{Neu} is stronger than \eqref{convex} and entails it with $c_1=c_2=\frac14)$. 

The analysis in the oscillating case on which we focus from now on, $\eta\neq$\,\textsl{triv}, is much more delicate and relies on explicit realizations of 
\eqref{first}. In the orthogonal root system case discussed in Section \ref{sec:ort}, which is a prelude to the much more involved cases of $I_2(3)$ 
and $I_2(4)$, $p_t^{\eta,\R_{k,+}}$ is given explicitly by \eqref{ort1}, and the analysis is relatively straightforward. In the case of the dihedral 
root system  $I_2(4)$, $p_t^{\eta,C_+^{(4)}}$ is given explicitly in terms of the hyperbolic sines and cosines by formulas stated in the beginning of Section  
\ref{sec:dih} for three  nontrivial homomorphisms $\eta$. For two of them, different from $\eta=$\,\textsl{det}, the analysis, performed in Sections 
\ref{sec:as2} and \ref{sec:as3}, is again relatively simple due to the fact that the relevant formulas are explicitly positive. This is not the 
case of $\eta=$\,\textsl{det}, where the delicacy represented by an oscillation has to be overcame by different means. The proof of Theorem \ref{thm:2} 
is contained in Section \ref{sec:es}; the case $m=4$ is treated first since it is slightly easier than that for $m=3$. Section \ref{sec:dih2} is devoted 
to the analysis of the $I_2(3)$ case and the proof of Theorem \ref{thm:1} is contained there. Finally,  Section \ref{sec:app}, the appendix 
consisting of four parts, is devoted to: (A) a discussion of harmonic profiles in the general setting of $C_+^{(m)}$; also a conjecture related 
to $C_+^{(m)}$ is stated there; (B) an application of the obtained bounds to a new setting of intersections of two half-spaces in $\R^d$, $d\ge3$; 
(C) an analysis of relations between our genuine estimates and bounds existing in the literature; (D) proofs of some postponed technicalities. 

Without explicit formulas for $p_t^{{\rm det},\,C_+}$ (that come from \eqref{first}) it is hard to expect an approach allowing to establish 
genuinely  sharp estimates for general $C_+$. Even for the simplest case of non-orthogonal root systems $I_2(m)$, $m\ge3$, i.e. for
the Weyl chamber $C_+^{(m)}$, being the cone on the plane with aperture $\pi/m$, proving genuinely  sharp estimates for the Dirichlet kernel 
$p_t^{{\rm det},\,C_+^{(m)}}$ seems to be, except for the cases $m\neq3$ and $m\neq4$, a challenging problem. Finally, we would like to emphasize that all our methods are purely analytic. 

Historical comments on Gaussian bounds for heat kernels on Euclidean (and more general) domains are scattered throughout \cite{GSC}; this monograph contains
also an extensive bibliography of the subject. The survey by Saloff-Coste \cite{SC} should also be consulted. 

We use the notation  $X\simeq Y$ to indicate that positive quantities $X$ and $Y$ are \textit{(uniformly) equivalent}, which means that  $C^{-1}Y\le X\le CY$ 
for some constant $C>0$, independent of significant quantities entering $X$ and $Y$. Writing $X\simeq\simeq Y\exp(-cZ)$ will mean that there exist positive constants $C,c_1,c_2$, such that 
$$
C^{-1}Y\exp(-c_1Z)\le X\le CY\exp(-c_2Z).
$$
This type of estimate we call \textit{qualitatively sharp}, whereas in the case $c_1=c_2$ we shall use the term \textit{genuinely sharp} 
estimate; sometimes \textit{genuinely} is replaced by \textit{quantitatively}. 

\section{Preliminaries} \label{sec:prel}
In this short section we collect all notions which are necessary for further reading. For details the reader can consult \cite{S2} and the references therein. 

The notion of finite reflection group (\textit{reflection group} for short) is based on the concept of a root system, a remarkably deep concept, with many ramifications in algebra and analysis. A \textit{root system} is a finite set $R$ of nonzero vectors in $\Rd$, called 
\textit{roots}, such that $ s_\a(R)=R$ for every $\a\in R$, where 
$$
s_\a(x)=x-\frac{2\langle \a,x\rangle}{\langle \a,\a\rangle}\a,  \qquad x\in\Rd,
$$
is the orthogonal reflection in $\langle\a\rangle^\bot=\{x\in\Rd\colon \langle x,\a\rangle=0\}$, the hyperplane orthogonal to $\a$. Clearly, 
$\langle-\a\rangle^\bot=\langle\a\rangle^\bot$, and to avoid further repetitions of these hyperplanes it is additionally assumed that 
$R\cap\mathbb{R}\a=\{\a,-\a\}$. The \textit{reflection group associated with} $R$ is the subgroup of $O(d)$ generated by the reflections $s_\a$, 
$W=W(R)=\langle s_\a\colon \a\in R\rangle$. The set $\Rd\setminus\bigcup_{\a\in R}\langle \a\rangle^\bot$ splits into an even number (equal to $|W|$)
of connected mutually congruent open components called \textit{Weyl chambers}. A choice of $x_0\in \Rd$ such that  $\langle \a,x_0\rangle\neq0$ 
for every $\a\in R$, gives the partition $R=R_+\sqcup (-R_+)$, where $R_+=\{\a\in R\colon \langle \a,x_0\rangle>0\}$ is referred to as the \textit{set of positive roots}. The partition distinguishes the chamber $C_+=\{x\in\Rd\colon \forall\, \a\in R_+ \,\,\,\langle x,\a\rangle>0\}$. 
Geometrically, as an intersection of a finite number of open half-spaces with supporting hyperplanes passing through the origin, $C_+$ is an \textit{open polyhedral cone} in $\Rd$ with vertex at the origin. Obviously, $C_+$ is an unbounded convex domain. 

In dimension 2 any root system is isomorphic (in the sense that the sets of the corresponding reflection lines are identical up to a rotation) 
to $I_2(m)$ for some $m\ge1$, where $I_2(m)=\{z_j\colon j=0,\ldots,2m-1\}$ in $\R^2\simeq \C$ and  $z_j=e^{i\pi j/m}$.
The reflection group $\mathcal D_m=W(I_2(m))$ is called the \textit{dihedral group} (frequently with restriction $m\ge3$). Geometrically, 
for $m\ge3$ this is the group of isometries of the regular $m$-gon centered at the origin with one of the vertices located at $i\simeq(0,1)\in\R^2$. 

In \cite{S2} realizations of the Laplacian on Weyl chambers associated with a reflection group were investigated; 
the more general setting of an arbitrary reflection group $W$ and an arbitrary open $W$-invariant set $\Omega\subset\Rd$ was treated by the author in \cite{S1}. 
These realizations in $L^2(C_+)$ are labeled by homomorphisms $\Hom$, are denoted $-\Delta^+_\eta$ and called the $\eta$-Laplacians on $C_+$. 
Heuristically, $\Hom$ serves for assigning signs to the facets $\mathcal F_\a=\overline{C_+}\cap\langle\a\rangle^\bot$ of $C_+$, which are the parts of
the boundary of $C_+$,   by setting ${\rm sign}_\eta(\mathcal F_\a):= \eta(s_\a)$, $\a\in\Sigma$, where $\Sigma\subset R_+$ is the (unique) system of 
simple roots in $R_+$.

Each  $-\Delta^+_\eta$ is a nonnegative self-adjoint extension of the differential operator $f\mapsto -\Delta f$, the (minus) Laplacian on $C_+$, with 
dense in $L^2(C_+)$ domain $C^\infty_c(C_+)$. The Dirichlet and Neumann Laplacians on $C_+$ are included and correspond to $\eta=$\,\textsl{det} and $\eta\equiv1$, 
respectively, and the other $\eta$-Laplacians are in between. The $\eta$-heat semigroup $\{e^{-t(-\Delta^+_\eta)}\}_{t>0}$ consists of integral operators with 
kernels $\{p_t^{\eta,\,C_+}\}_{t>0}$. 
It was proved, see \cite[Theorem 1.1]{S1} and \cite[Corollary 1.2]{S1}, that the $\eta$-heat kernel on $C_+$, $\{p_t^{\eta,C_+}\}_{t>0}$, is given by 
\begin{equation}\label{first}
p_t^{\eta,\,C_+}(x,y)=\sum_{g\in W}\eta(g)p_t^{(d)}(gx-y),\qquad x,y\in C_+,
\end{equation}
and its positivity on $C_+$ was established in \cite[Theorem 3.4]{S2}. On the boundary of $C_+$, $p_t^{\eta,\,C_+}$ behaves as follows, see 
\cite[Proposition 2.2]{S2}: for fixed $y\in C_+$ and $\a\in\Sigma$, if ${\rm sign}_\eta(\mathcal F_\a)=-1$, then $p_t^{\eta,\,C_+}(\cdot,y)$ vanishes on 
$\mathcal F_\a$, while if ${\rm sign}_\eta(\mathcal F_\a)=1$, then $p_t^{\eta,\,C_+}(\cdot,y)$ is positive on the $(d-1)$-dimensional interior of $\mathcal F_\a$. Finally we mention that the symmetry of $p_t^{\eta,\,C_+}(x,y)$ in $x$ and $y$ is seen directly from \eqref{first}. 

For reader's convenience we collect here the following formulas used throughout:
\begin{equation}\label{sin-sin}
\sinh \a\sinh \b=\frac12\big(\cosh(\a+\b)-\cosh(\a-\b)\big), 
\end{equation}
\begin{equation}\label{sin-cos}
\sinh \a\cosh \b=\frac12\big(\sinh(\a+\b)+\sinh(\a-\b)\big),
\end{equation}
\begin{equation}\label{cos-cos}
\cosh \a\cosh \b=\frac12\big(\cosh(\a+\b)+\cosh(\a-\b)\big).
\end{equation}

\section{Orthogonal root system case} \label{sec:ort}

An  introduction to this section can be found in \cite[Section 4.1]{MS}, \cite[Section 6.1]{S1}, and \cite[Section 5.1]{S2}. We recall basic 
facts for reader's convenience.

Up to a rotation, any orthogonal root system in $\Rd$ is isomorphic to the system $\{\pm e_{j_1},\ldots,\pm e_{j_k}\}$, where $1\le j_1<j_2<\ldots<j_k=d$ 
and $1\le k\le d$. Rather than to treat orthogonal systems in such generality it is convenient to consider slightly more ordered version, namely 
$R_k:=\{\pm e_{d-k+1},\ldots,\pm e_d\}$; in any case this limitation does not reduce generality of obtained results. Then, with the choice 
$R_{k,+}:=\{e_{d-k+1},\ldots,e_d\}$, the corresponding distinguished Weyl chamber is $\R_{+,k}=\mathbb{R}^{d-k}\times(0,\infty)^k$; for $d\ge2$ and $k=1$ 
this is  the upper half-space in $\Rd$, for $d=k=1$ this is the half-line $(0,\infty)$. 

Moreover,$\,W(R_k)\simeq\widehat{\mathbb Z}_2^k$ (group isomorphism) and the action of any $\ve=(\ve_j)_{j=1}^k\in \widehat{\mathbb Z}_2^k=\{1,-1\}$ 
on $\Rd$ is through
$$
x\to \ve x=(x_1,\ldots,x_{d-k},\ve_1x_{d-k+1},\ldots, \ve_k x_d).
$$
Consequently, we identify ${\rm Hom} (\widehat{\mathbb Z}_2^k, \widehat{\mathbb Z}_2)$ with $\mathbb Z_2^k$, where this time, $\mathbb Z_2=\{0,1\}$ with 
addition modulo 2. In this identification a homomorphism $\eta \in \mathbb Z_2^k$ of the reflection group represented by $\widehat{\mathbb Z}_2^k$ into 
$\widehat{\mathbb Z}_2$ acts  through $\ve \to \ve^\eta:=\prod_{j=1}^k \ve_j^{\eta_j}$ for  $\ve=(\ve_j)_{j=1}^k$. The trivial homomorphism is represented 
by $\textbf{0}:=(0,\ldots,0)$. The other distinguished homomorphism is represented by $\textbf{1}:=(1,\ldots,1)$, in the general case denoted as \textsl{det}. 

Thus, for $x,y\in \R_{+,k}$ and $\eta\in\Z_2^k$ we have (see \cite[Proposition 2.4]{S2}, cf. also \cite[Corollary 6.1]{S1})
\begin{equation}\label{ort1}
p_t^{\eta,\R_{k,+}}(x,y)=\prod_{j=1}^{d-k}p_t^{(1)}(x_{j}-y_{j})\prod_{j=d-k+1}^{d}\Big(p_t^{(1)}(x_j-y_j)+(-1)^{\eta(j-d+k)}p_t^{(1)}(x_j+y_j)\Big).
\end{equation}
By convention, the first product is assumed to be 1 if $k=d$. Analogous convention will be used in similar occurrences. 
Notice that the right-hand side  of \eqref{ort1} correctly reflects the behavior of $p_t^{\eta,\R_{+,k}}$  on the boundary of $\R_{+,k}$: given $y\in \R_{+,k}$, 
$p_t^{\eta,\R_{+,k}}(x,y)$ vanishes if $x$  belongs to any of the facets $\calF_j:=\langle e_{d-k+j}\rangle^\bot\cap\overline{\R_{+,k}}$, $j\in J_\eta$ (if any), 
and $p_t^{\eta,\R_{+,k}}(x,y)$ is positive if $x$ belongs to the $(d-1)$-dimensional interior of $\calF_j$ with $j\in J_\eta^c$ (if any), 
the complement of $J_\eta$ in $\{1,\ldots,k\}$. 

We are now ready to prove Proposition \ref{prop:1}. In dimension one, $d=1$, for the Neumann and Dirichlet heat kernels on $(0,\infty)$, 
$p_t^N(x,y)=p_t^{(1)}(x-y)+p_t^{(1)}(x+y)$ and $p_t^D(x,y)=p_t^{(1)}(x-y)-p_t^{(1)}(x+y)$, we have for $x,y>0$ and $t>0$,  
$p_t^N(x,y)\simeq p_t^{(1)}(x-y)$ and $p_t^D(x,y)\simeq \frac{xy}{xy+t}p_t^{(1)}(x-y)$; the latter is a consequence of $1-e^{-u}\simeq \frac u{u+1}$ for $u>0$. 
In dimension $d\ge2$, due to the product structure of $p_t^{\eta,\R_{+,k}}(x,y)$ (which is a general consequence of reducibility of the orthogonal 
root system  $R_k$), the proof is a simple  tensorization of the one-dimensional result. 
Indeed, consider just for the simplicity of notation only $\eta=\textbf{1}$ (the case of the Dirichlet kernel, in the case of general $\eta$ 
the changes are routine). We obtain for $x=(x',x_{d-k+1},\ldots,x_{d})$, $y=(y',y_{d-k+1},\ldots,y_{d})$, and $t>0$,
\begin{align*} 
p_t^{\textbf{1},\R_{k,+}}(x,y)&=p_t^{(d-k)}(x'-y')\prod_{j=d-k+1}^{d}\Big(p_t^{(1)}(x_j-y_j)-p_t^{(1)}(x_j+y_j)\Big)\\
&=p_t^{(d)}(x-y)\prod_{j=d-k+1}^{d}\Big(1-e^{-x_jy_j/t}\Big)\\
&\simeq p_t^{(d)}(x-y)\prod_{j=d-k+1}^{d}\frac{x_jy_j}{x_jy_j+t}.
\end{align*}

\begin{remark}
For $\eta=\textbf{0}$ (the case of the Neumann heat kernel), when $J_{\textbf{0}}=\emptyset$, the product in \eqref{orto} equals $1$, and the result of Proposition \ref{prop:1} is consistent with \eqref{Neu} stated for the general case of $\eta=$\,\textsl{triv} and $C_+$. For $\eta=\textbf{1}$ (the case of the Dirichlet heat kernel), when $J_{\textbf{1}}=\{1,\ldots,k\}$, the estimates \eqref{orto} are stronger than \eqref{convex2}. Indeed, the harmonic profile of 
$\R_{+,k}$ is
$$
h(y)=\prod_{j=d-k+1}^d y_{j}, \qquad y\in \R_{+,k},
$$
and 
\begin{equation*}
V_{h^2}(x,r)=\int_{\{y\in \R_{+,k}\colon \|x-y\|<r\}} h^2(y)\,dy
\simeq r^d\prod_{j=d-k+1}^d (x_j^2+r^2)
\end{equation*}
uniformly in $x\in \R_{+,k}$ and $r>0$, so that \eqref{convex2} in the considered setting becomes
\begin{align*}
\prod_{j=d-k+1}^d \frac{x_jy_j}{\sqrt{(x_j^2+t)(y_j^2+t)}}\,t^{-d/2} e^{-c_1\frac{\|x-y\|^2}{t}}&\lesssim p_t^{\textbf{1},\R_{+,k}}(x,y)
\lesssim\\
& \prod_{j=d-k+1}^d \frac{x_jy_j}{\sqrt{(x_j^2+t)(y_j^2+t)}}\,t^{-d/2}e^{-c_2\frac{\|x-y\|^2}{t}}.
\end{align*}
To see that the above is implied by \eqref{orto} taken with $\eta=\textbf{1}$ note, that the left-hand side follows with $c_1=\frac14$ from \eqref{ort1} 
taken with $\eta=\textbf{1}$ by using the simple inequality $(ab+1)^2\le (a^2+1)(b^2+1)$, $a,b>0$. On the other hand,  the right-hand 
side is implied, with any $0<c_2<\frac14$, by \eqref{orto} and the inequality
\begin{equation}\label{ort2}
(a+1)(b+1)\lesssim (ab+1)e^{\varepsilon(a-b)^2}, \qquad a,b>0.
\end{equation}
\end{remark}

\begin{remark}\label{long}
It follows from Proposition \ref{prop:1} what the long-time decay of $p_t^{\eta,\R_{+,k}}$ is. Namely, given $x,y\in\R_{+,k}$ and $\eta\in \mathbb Z_2^k$ 
we have $p_t^{\eta,\R_{+,k}}(x,y)\simeq t^{-d/2-\#J_\eta}$ for $t\ge1$. As it is seen, the decay improves when the number of facets of $\R_{+,k}$,  where the Dirichlet type boundary conditions are imposed increases. 
\end{remark}

Finally, we point out an inconsistency between a result appearing in the literature and \eqref{ort1} taken with $\eta=\textbf{1}$ and $k=1$, i.e. 
the case of the Dirichlet heat kernel $p_t^D:=p_t^{\eta,\R_{+,1}}$ on the half-space $\R_{+,1}=\R^{d-1}\times(0,\infty)$ as a Weyl chamber. 
Namely, Proposition \ref{prop:1} gives 
\begin{equation}\label{ort3}
p_t^D(x,y)\simeq \frac{x_dy_d}{x_dy_d+t}\,p_t^{(d)}(x-y),
\end{equation}
uniformly in $t>0$ and $x=(x',x_d)\in \R^{d-1}\times(0,\infty)$, $y=(y',y_d)\in \R^{d-1}\times(0,\infty)$. 
It follows that the equivalence stated in \cite[Section 5]{SC}, 
\begin{equation}\label{ort4}
p_t^D(x,y)\simeq \frac{x_dy_d}{(x_d+\sqrt t)(y_d+\sqrt t)}\,p_t^{(d)}(x-y),
\end{equation}
is not consistent with \eqref{ort3} (however, qualitatively sharp estimates in \cite[Section 1.2]{GSC} are correct). Indeed, if \eqref{ort4} were correct, 
then, with the simplifying case of $t=1$, one would get $(x_d+1)(y_d+1)\lesssim x_dy_d+1$ for $x_d,y_d>0$, which is false.

\section{Dihedral root system $I_2(4)$ case} \label{sec:dih}
For an introduction to this section see \cite[Section 5.4]{S2} and \cite[Section 6.2]{S1}. We only recall that the open cone on the plane 
$$
C_+^{(4)}=\{\rho e^{i \theta}\colon \rho>0,\,\,\,\,0<\theta<\pi/4\}=\{(x_1,x_2)\colon 0<x_2<x_1\},
$$
with vertex at the origin and aperture $\pi/4$, is the distinguished Weyl chamber in this context.

The kernels $p_t^{\eta,C_+^{(4)}}$ associated with one of the four homomorphisms $\eta$ are: for $t>0$ and $x,y\in C_+^{(4)}$, 
$$
p_t^{\eta,C_+^{(4)}}(x,y)=\frac1{2\pi t}\exp\big(-\frac{\|x-y\|^2}{4t}\big)\exp\big(-\frac{\langle x,y\rangle}{2t}\big)\Psi_{t,\eta}(x,y),
$$
where (we write $\Psi_{t,\calN_i}$, $i=1,2$, in place of $\Psi_{t,\eta_{\calN_i}}$)
\begin{align*} 
(1)\quad \eta&=\sgn: \qquad  \Psi_{t,\sgn}(x,y)=\sinh\big(\frac{x_1y_1}{2t}\big)\sinh\big(\frac{x_2y_2}{2t}\big)-
\sinh\big(\frac{x_2y_1}{2t}\big)\sinh\big(\frac{x_1y_2}{2t}\big),\\
(2)\quad \eta&=\eta_{\calN_1}: \qquad   \Psi_{t,\calN_1}(x,y)=\cosh\big(\frac{x_1y_1}{2t}\big)\cosh\big(\frac{x_2y_2}{2t}\big)-
\cosh\big(\frac{x_2y_1}{2t}\big)\cosh\big(\frac{x_1y_2}{2t}\big),\\
(3)\quad \eta&=\eta_{\calN_2}: \qquad  \Psi_{t,\calN_2}(x,y)=\sinh\big(\frac{x_1y_1}{2t}\big)\sinh\big(\frac{x_2y_2}{2t}\big)+
\sinh\big(\frac{x_2y_1}{2t}\big)\sinh\big(\frac{x_1y_2}{2t}\big),\\
(4)\quad \eta&={\rm triv}: \qquad  \Psi_{t,{\rm triv}}(x,y)=\cosh\big(\frac{x_1y_1}{2t}\big)\cosh\big(\frac{x_2y_2}{2t}\big)+
\cosh\big(\frac{x_2y_1}{2t}\big)\cosh\big(\frac{x_1y_2}{2t}\big).
\end{align*}
We mention that $\eta_{\calN_1}$ corresponds to the case of the Laplacian on $C_+^{(4)}$ subject to the Dirichlet boundary condition on 
$\{(x_1,x_1)\colon x_1\ge0\}$ and the Neumann boundary condition on $\{(x_1,0)\colon x_1>0\}$,  and for $\eta_{\calN_2}$ the roles of these 
boundary conditions are switched. 

The formula
\begin{equation}\label{cosh}
\cosh \a-\cosh \b=2\sinh\big(\frac{\a+\b}2\big)\sinh\big(\frac{\a-\b}2\big),
\end{equation}
and \eqref{cos-cos} easily lead to an equivalent version of (2), namely, 
\begin{align*} 
(2)'\quad \Psi_{t,\calN_1}(x,y)=&\sinh\big((x_1+x_2)(y_1+y_2)/4t\big)\sinh\big((x_1-x_2)(y_1-y_2)/4t\big)\\
+&\sinh\big((x_1+x_2)(y_1-y_2)/4t\big)\sinh\big((x_1-x_2)(y_1+y_2)/4t\big),
\end{align*}
and this directly shows positivity of $\Psi_{t,\calN_1}$ (and thus also positivity of $p_t^{\eta_{\calN_1},C_+^{(4)}}$)  on $C_+^{(4)}$ (Remark: 
positivity of $\Psi_{t,\calN_1}$ on $C_+^{(4)}$ based on (2) was shown by a direct computation in \cite[p.\,369, B)]{S2}.) Notice that similar tools lead to 
\begin{align*} 
(1)'\quad \Psi_{t,\sgn}(x,y)=&\sinh\big((x_1+x_2)(y_1+y_2)/4t\big)\sinh\big((x_1-x_2)(y_1-y_2)/4t\big)\\
-&\sinh\big((x_1+x_2)(y_1-y_2)/4t\big)\sinh\big((x_1-x_2)(y_1+y_2)/4t\big),
\end{align*}
which is, however, not sufficient to claim directly the positivity of $\Psi_{t,\sgn}$ on $C_+^{(4)}$. (It is doubtful if $\Psi_{t,\sgn}$ could be expressed 
as a finite sum of explicitly positive summands.)

We remark that positivity of $\Psi_{t,\sgn}$ on $C_+^{(4)}$ based on (1) was shown by a direct computation in \cite[p.\,369, A)]{S2}. For an easier argument  
for this fact see a comment following Lemma \ref{lem:bas} .

It is worth mentioning that, according to the general principles (see \cite[(3.3)]{S2}), $p_t^{\sgn,C_+^{(4)}}\le p_t^{\eta,C_+^{(4)}}\le p_t^{{\rm triv},C_+^{(4)}}$ 
for $\eta=\eta_{\calN_i}$, $i=1,2$, which is equivalent to analogous inequalities for the corresponding $\Psi_{t,\eta}$ functions; these inequalities  are seen directly. 

We shall use the asymptotic 
\begin{equation}\label{ass}
\sinh u\simeq \frac u{u+1}e^u, \qquad 0<u<\infty.
\end{equation}

\subsection{Estimates of $p_t^{\eta_{\calN_2},C_+^{(4)}}$} \label{sec:as2}

\begin{proposition} \label{pro:N2}
We have, uniformly in $x=(x_1,x_2)\in C_+^{(4)},\,y=(y_1,y_2)\in C_+^{(4)}$, and $t>0$, 
\begin{equation}\label{XY}
p_t^{\eta_{\calN_2},C_+^{(4)}}(x,y)\simeq \frac{x_1y_1}{x_1y_1+t}\frac{x_2y_2}{x_2y_2+t}\,p_t^{(2)}(x-y).
\end{equation}
\end{proposition}

\begin{proof}
By a simple homogeneity argument it suffices to establish \eqref{XY} for $t=1/2$ and hence we are reduced to verifying that for $0<x_2<x_1$ and $0<y_2<y_1$ 
it holds 
\begin{align*}
e^{-\langle x,y\rangle}\Psi_{1/2,\calN_2}(x,y)&=e^{-\langle x,y\rangle}\big(\sinh(x_1y_1)\sinh(x_2y_2)+\sinh(x_2y_1)\sinh(x_1y_2) \big)\\
&\simeq \frac{x_1y_1}{x_1y_1+1}\frac{x_2y_2}{x_2y_2+1}
\end{align*}
(replacing the summand $1$ by $1/2$, or vice versa, in the denominators of the last two factors is justified). We now use \eqref{ass} to obtain 
\begin{align*}
&\,\,\,e^{-\langle x,y\rangle}\Psi_{1/2,\calN_2}(x,y)\\
&\simeq e^{-x_1y_1-x_2y_2}\Big(\frac{x_1y_1}{x_1y_1+1}e^{x_1y_1}\,\frac{x_2y_2}{x_2y_2+1}e^{x_2y_2}+
\frac{x_2y_1}{x_2y_1+1}e^{x_2y_1}\,\frac{x_1y_2}{x_1y_2+1}e^{x_1y_2} \Big)\\
&=\frac{x_1y_1}{x_1y_1+1}\,\frac{x_2y_2}{x_2y_2+1}+\frac{x_2y_1}{x_2y_1+1}\,\frac{x_1y_2}{x_1y_2+1}\,e^{-(x_1-x_2)(y_1-y_2)}\\
&=\frac{x_1y_1}{x_1y_1+1}\,\frac{x_2y_2}{x_2y_2+1}\Big(1+\frac{x_1y_1+1}{x_2y_1+1}\,\frac{x_2y_2+1}{x_1y_2+1}\, e^{-(x_1-x_2)(y_1-y_2)} \Big).
\end{align*}
To finish, it is sufficient to check that 
$$
\phi_2(x,y):=\frac{x_1y_1+1}{x_2y_1+1}\,\frac{x_2y_2+1}{x_1y_2+1}\, e^{-(x_1-x_2)(y_1-y_2)}
$$
is bounded from above for $0<x_2<x_1$ and $0<y_2<y_1$; these constraints remain in force until the end of the proof. 

Assume that $\frac{x_1}2\le x_2<x_1$, which implies $x_2y_1+1>(x_1y_1+1)/2$, and we have
$$
\phi_2(x,y)<\frac{x_1y_1+1}{x_2y_1+1}<2.
$$
Analogously, in the symmetric case $\frac{y_1}2\le y_2<y_1$ we obtain
$$
\phi_2(x,y)<\frac{x_1y_1+1}{x_1y_2+1}<2.
$$
Therefore we are left with the case when $0<x_2<\frac{x_1}2$ and $0<y_2<\frac{y_1}2$, i.e. when $x_1-x_2>\frac12x_1$ and  $y_1-y_2>\frac12y_1$. Then
$$
\phi_2(x,y)<\frac{x_1y_1+1}{x_2y_1+1}\, e^{-\frac14x_1y_1}
$$
and if $x_1y_1\le2$, then $\phi_2(x,y)\le 3$, while if  $x_1y_1>2$, then using $ue^{-u}\lesssim 1$ for $u>0$, gives
$$
\phi_2(x,y)\lesssim \frac{x_1y_1}{x_2y_1+1}\, \frac 1{x_1y_1}<1
$$
and we are done.
\end{proof}
Notice that the right-hand side  of \eqref{XY} correctly reflects the behavior of $p_t^{\eta_{\calN_2},C_+^{(4)}}(x,y)$ on the boundary of $C_+^{(4)}$: 
for $y\in C_+^{(4)}$, $p_t^{\eta_{\calN_2},C_+^{(4)}}(\cdot,y)$ vanishes on the half-line $\{(s,0)\colon s\ge0\}$, and is positive on the half-line 
$\{(s,s)\colon s>0\}$. 

\subsection{Estimates of $p_t^{\eta_{\calN_1},C_+^{(4)}}$} \label{sec:as3}
\begin{proposition} \label{pro:N1}
We have, uniformly in  $x=(x_1,x_2)\in C_+^{(4)}$, $y=(y_1,y_2)\in C_+^{(4)}$, and $t>0$, 
\begin{equation}\label{N1}
p_t^{\eta_{\calN_1},C_+^{(4)}}(x,y)\simeq \frac{x_1y_1}{x_1y_1+t}\,\frac{(x_1-x_2)(y_1-y_2)}{(x_1-x_2)(y_1-y_2)+t}\,p_t^{(2)}(x-y).
\end{equation}
\end{proposition}
\begin{proof}
It suffices to establish \eqref{N1} for $t=1/2$ and hence the task reduces to verifying that for $0<x_2<x_1$ and $0<y_2<y_1$ we have 
\begin{align*}
e^{-\langle x,y\rangle}\Psi_{1/2,\calN_1}(x,y)&=e^{-\langle x,y\rangle}\Big(\sinh\big((x_1+x_2)(y_1+y_2)/2\big)\sinh\big((x_1-x_2)(y_1-y_2)/2\big)\\
&\,\,\,\,\,\,\,\,\,\,\,\,\,\,\,\,\,\,\,\,\,\,\,\,+\sinh\big((x_1+x_2)(y_1-y_2)/2\big)\sinh\big((x_1-x_2)(y_1+y_2)/2\big)\Big)\\
&\simeq \frac{x_1y_1}{x_1y_1+1}\,\frac{(x_1-x_2)(y_1-y_2)}{(x_1-x_2)(y_1-y_2)+1}. 
\end{align*}
We use \eqref{ass} to obtain that $e^{-\langle x,y\rangle}\Psi_{1/2,\calN_1}(x,y)$ is  equivalent to 
\begin{align*}
e^{-x_1y_1-x_2y_2}&\Big(\frac{(x_1+x_2)(y_1+y_2)}{(x_1+x_2)(y_1+y_2)+1}\,e^{(x_1+x_2)(y_1+y_2)/2}\,
\frac{(x_1-x_2)(y_1-y_2)}{(x_1-x_2)(y_1-y_2)+1}\,e^{(x_1-x_2)(y_1-y_2)/2}\\
&+\frac{(x_1+x_2)(y_1-y_2)}{(x_1+x_2)(y_1-y_2)+1}\,e^{(x_1+x_2)(y_1-y_2)/2}\,\frac{(x_1-x_2)(y_1+y_2)}{(x_1-x_2)(y_1+y_2)+1}\,e^{(x_1-x_2)(y_1+y_2)/2}\Big)\\
&=\frac{(x_1+x_2)(y_1+y_2)}{(x_1+x_2)(y_1+y_2)+1}\,\frac{(x_1-x_2)(y_1-y_2)}{(x_1-x_2)(y_1-y_2)+1}\\
&+\frac{(x_1+x_2)(y_1-y_2)}{(x_1+x_2)(y_1-y_2)+1}\,\frac{(x_1-x_2)(y_1+y_2)}{(x_1-x_2)(y_1+y_2)+1}\,e^{-2x_2y_2}\\
&=\frac{(x_1+x_2)(y_1+y_2)}{(x_1+x_2)(y_1+y_2)+1}\,\frac{(x_1-x_2)(y_1-y_2)}{(x_1-x_2)(y_1-y_2)+1}\\
&\times \Big(1+\frac{(x_1+x_2)(y_1+y_2)+1}{(x_1+x_2)(y_1-y_2)+1}\,\frac{(x_1-x_2)(y_1-y_2)+1}{(x_1-x_2)(y_1+y_2)+1}\, e^{-2x_2y_2} \Big).
\end{align*}
Observing that
$$
\frac{(x_1+x_2)(y_1+y_2)}{(x_1+x_2)(y_1+y_2)+1}\simeq\frac{x_1y_1}{x_1y_1+1},
$$
to finish it suffices now to check that 
$$
\phi_1(x,y):=\frac{(x_1+x_2)(y_1+y_2)+1}{(x_1+x_2)(y_1-y_2)+1}\,\frac{(x_1-x_2)(y_1-y_2)+1}{(x_1-x_2)(y_1+y_2)+1}\, e^{-2x_2y_2}\lesssim1,
$$
uniformly in $0<x_2<x_1$ and $0<y_2<y_1$.

We have $(x_1+x_2)(y_1+y_2)+1< 2\big((x_1+x_2)y_1+1\big)$ and for $0<y_2<\frac{y_1}2$,  $(x_1+x_2)(y_1-y_2)+1>\frac12\big((x_1+x_2)y_1+1\big)$, hence
in this case 
$$
\phi_1(x,y)< \frac{(x_1+x_2)(y_1+y_2)+1}{(x_1+x_2)(y_1-y_2)+1}<4.
$$
Analogously, for $0<x_2<\frac{x_1}2$ we end up with the situation symmetric to the just considered and 
$$
\phi_1(x,y)< \frac{(x_1+x_2)(y_1+y_2)+1}{(x_1-x_2)(y_1+y_2)+1}<4.
$$
Therefore it remains to analyse the case when  $\frac{x_1}2\le x_2<x_1$ and $\frac{y_1}2\le y_2<y_1$. We have
$$
\phi_1(x,y)<\frac{(x_1+x_2)(y_1+y_2)+1}{(x_1+x_2)(y_1-y_2)+1}\,e^{-2x_2y_2}<4(x_1y_1+1)\, e^{-2x_2y_2}. 
$$
Now, if $x_1y_1\le1$, then $\phi_1(x,y)<8$. In the opposite case when $x_1y_1>1$, since $x_2y_2\simeq x_1y_1$, we obtain
$$
\phi_1(x,y)\lesssim \frac {x_1y_1+1}{x_2y_2}\lesssim \frac{x_1y_1+1}{x_1y_1}\lesssim 1.
$$
\end{proof}
Again, we remark that  the right-hand side  of \eqref{N1} correctly reflects the behavior of $p_t^{\eta_{\calN_1},C_+^{(4)}}$ on  $\partial C_+^{(4)}$: 
for $y\in C_+^{(4)}$, $p_t^{\eta_{\calN_1},C_+^{(4)}}(\cdot,y)$ vanishes on  $\{(s,s)\colon s\ge0\}$, and is positive on  $\{(s,0)\colon s>0\}$. 

\subsection{Estimates of $p_t^{{\rm det},C_+^{(4)}}$} \label{sec:es}

This section is entirely devoted to the proof of Theorem \ref{thm:2}. Recall that 
$$
p_t^{{\rm det},C_+^{(4)}}(x,y)=e^{-\langle x,y\rangle /2t}\Psi_{t,\sgn}(x,y)p_t^{(2)}(x-y)
$$ 
and $e^{-\langle x,y\rangle /2t}\Psi_{t,\sgn}(x,y)$ equals to 
$$
e^{-\langle x,y\rangle/2t}\\\Big(\sinh\big(\frac{x_1y_1}{2t}\big)\sinh\big(\frac{x_2y_2}{2t}\big)-
\sinh\big(\frac{x_2y_1}{2t}\big)\sinh\big(\frac{x_1y_2}{2t}\big)\Big).
$$
Our aim is to prove that the above expression is equivalent to the product of the four factors preceding $p_t^{(2)}(x-y)$ on the right-hand side of \eqref{sgn}. Obviously we can reduce our consideration to $t=1/2$ (exchanging $1/2$ with 1 in the four mentioned factors is easily justified). Thus, assume $t=1/2$ and note that the first two factors on the right-hand side  of \eqref{sgn} may be naturally singled out from the considered expression. This is because 
\begin{align*}
e^{-\langle x,y\rangle}\Psi_{1/2,\sgn}(x,y)
&=e^{-\langle x,y\rangle}\sinh\big(x_1y_1\big)\sinh\big(x_2y_2\big)
\Big(1-\frac{\sinh\big(x_1y_2\big)\sinh\big(x_2y_1\big)}{\sinh\big(x_1y_1\big)\sinh\big(x_2y_2\big)}\Big)\\
&=\frac14\big(1-e^{-2x_1y_1}\big)\big(1-e^{-2x_2y_2}\big)\Big(1-\frac{\sinh\big(x_1y_2\big)\sinh\big(x_2y_1\big)}{\sinh\big(x_1y_1\big)\sinh\big(x_2y_2\big)}\Big)\\
&\simeq \frac{x_1y_1}{x_1y_1+1}\,\frac{x_2y_2}{x_2y_2+1}\Big(1-\frac{\sinh\big(x_1y_2\big)\sinh\big(x_2y_1\big)}{\sinh\big(x_1y_1\big)\sinh\big(x_2y_2\big)}\Big),
\end{align*}
where in the final step we used $1-e^{-u}\simeq \frac u{u+1}$.

It now remains to show, and this is the essential part of the proof, that
\begin{equation}\label{ort}
1-\frac{\sinh\big(x_1y_2\big)\sinh\big(x_2y_1\big)}{\sinh\big(x_1y_1\big)\sinh\big(x_2y_2\big)}\simeq 
\frac{(x_1-x_2)(y_1-y_2)}{(x_1-x_2)(y_1-y_2)+1}\frac{x_1y_1}{x_1y_1+1}, \qquad x,y\in C_+^{(4)}.
\end{equation}
For the analysis of the left-hand side  of \eqref{ort}, and for $x,y\in C_+^{(4)}$, it is convenient to introduce the variables 
$$
X=x_1y_1, \quad Y=x_2y_1, \quad s=\frac{y_2}{y_1},
$$ 
so that $0<s<1$, $0<Y< X$ and $(x_1-x_2)(y_1-y_2)=(1-s)(X-Y)$, and to consider the function
$$
G(s,X,Y):=\frac{\sinh X\sinh(sY)-\sinh(sX)\sinh Y}{\sinh X \sinh(sY)}=1-\frac{\sinh(sX)}{\sinh X}\frac{\sinh Y}{\sinh(sY)}
$$
which for $0<s<1$, $0<Y<X$, corresponds to the left-hand side  of \eqref{ort}. The proof of \eqref{ort} will rely on a thorough analysis of the function $G$. 

We first observe that the right-hand side  of \eqref{ort}  suggests the natural splitting of further reasoning into the following three cases that 
depend on the configuration of $(x_1-x_2)(y_1-y_2)$ and $x_1y_1$  with respect to 1. 

\noindent \textbf{Case (1)} $x_1y_1\le1$; then also $(x_1-x_2)(y_1-y_2)<1$ and the right-hand side of \eqref{ort} can be replaced by $(x_1-x_2)(y_1-y_2)x_1y_1$.

\noindent \textbf{Case (2)} $(x_1-x_2)(y_1-y_2)\le1$ and $x_1y_1> 1$; the right-hand side  of \eqref{ort} can be replaced by $(x_1-x_2)(y_1-y_2)$. 

\noindent \textbf{Case (3)} $(x_1-x_2)(y_1-y_2)>1$; then also $x_1y_1>1$ and  the right-hand side  of \eqref{ort} can be replaced by $1$. 

We will analyse these cases separately rephrasing the resulting claims in terms of the $s,X$ and $Y$ variables; below it is understood that the bounds hold uniformly in $0<s<1$, $Y>0$, and $X$ satisfying displayed conditions.

\noindent \textbf{Claim 1} 
\begin{equation}\label{cl3}
G(s,X,Y)\simeq (1-s)(X-Y)X, \qquad Y<X\le1, 
\end{equation}

\noindent \textbf{Claim 2} 
\begin{equation}\label{cl2}
G(s,X,Y)\simeq (1-s)(X-Y), \qquad \max(1,Y)<X\le Y+\frac1{1-s};
\end{equation}

\noindent \textbf{Claim 3} 
\begin{equation}\label{cl1}
G(s,X,Y)\simeq 1, \qquad X>Y+\frac1{1-s}.
\end{equation}

We shall use the following simple lemma.
\begin{lemma} \label{lem:bas} 
Let $s\in(0,1)$. The function $\frac{\sinh sx}{\sinh x}$ is decreasing on $(0,\infty)$. 
\end{lemma}
\begin{proof}
Obviously,
\begin{equation*}
\Big(\frac{\sinh sx}{\sinh x}\Big)'(x)=\frac{s\cosh(sx)\sinh x-\sinh(sx)\cosh x}{\sinh^2 x}<0\Leftrightarrow \frac{\tanh x}x<\frac{\tanh(sx)}{sx}.
\end{equation*}
The latter inequality follows since $\frac{\tanh u}u$ is decreasing on $(0,\infty)$. 
\end{proof}

Before commencing to the proofs of these three claims we list obvious properties of $G$ that could be easily translated onto the corresponding properties 
of the expression on the left-hand side  in \eqref{ort}. Firstly,
$$
G(s,X,Y)>0,\qquad 0<s<1,\quad 0<Y<X.
$$
This immediately follows from Lemma \ref{lem:bas}. (We note at this point that this entails the positivity of $p_t^{{\rm det},C_+^{(4)}}(x,y)$ on $C_+^{(4)}$, 
the result proved in  \cite{S2} by slightly more involved argument). Secondly, for the limiting values, we have
$$
G(s,Y^+,Y)=0=G(s,X,X^-), \qquad 0<s<1,\quad X>0,\quad Y>0,
$$
and
$$
G(0^+,X,Y)=0=G(1^-,X,Y), \qquad  0<Y<X.
$$
These properties are equivalent to the statement that $\Psi_{t,\sgn}(x,y)$ vanishes (and thus also $p_t^{{\rm det},C_+^{(4)}}(x,y)$ vanishes), 
if any of $x,y$ belongs to the boundary of $C_+^{(4)}$. Also, for any fixed $0<s<1$ and $Y>0$, $G(s,X,Y)$ increases as a function of $X$ on 
the $X$-interval $(Y,\infty)$ (nota bene, from $G(s,Y^+,Y)=0$ to $G(s,\infty,Y)=1$), and for any fixed $0<s<1$ and $X>0$, $G(s,X,Y)$ decreases 
on the $Y$-interval $(0,X)$ (from $G(s,X,0^+)=1-\frac1s\frac{\sinh(sX)}{\sinh X}$ to $G(s,X,X^-)=0$). This is  a consequence 
of the fact that $\frac{\sinh u}{\sinh su}$ increases for $u\in(0,\infty)$.

We are now ready to begin the proofs of our Claims. 

\noindent \textit{Proof of Claim 1}. 
For the denominator in the fraction defining $G(s,X,Y)$  we have 
$$
\sinh X\sinh\big(sY\big)\simeq sXY,
$$
so the assertion to be proved is, uniformly in $-1<s<1, 0<Y<X\le 1$, 
\begin{equation}\label{duo}
\sinh X\sinh\big(sY\big)-\sinh\big(sX\big)\sinh Y\simeq s(1-s)(X-Y)X^2Y.
\end{equation}
Using \eqref{sin-sin} and \eqref{cosh}, we rewrite the left-hand side  of \eqref{duo} to the form
$$
\sinh\Big((1+s)\frac{X+Y}2\Big)\sinh\Big((1-s)\frac{X-Y}2\Big)-\sinh\Big((1-s)\frac{X+Y}2\Big)\sinh\Big((1+s)\frac{X-Y}2\Big).
$$
Then, dividing both sides of accordingly modified \eqref{duo} by $(1+s)\frac{X+Y}2\,(1-s)\,\frac{X-Y}2$ and replacing (for simplicity 
of notation) $X/2$ and $Y/2$ just by $X$ and $Y$, respectively, brings the relation 
\begin{align}\label{duo2}
S(s,X,Y)&:=\frac{\sinh (1+s)(X+Y)}{(1+s)(X+Y)}\frac{\sinh (1-s)(X-Y)}{(1-s)(X-Y)}-
\frac{\sinh (1-s)(X+Y)}{(1-s)(X+Y)}\frac{\sinh(1+s)(X-Y)}{(1+s)(X-Y)}\nonumber\\
 &\,\,\,\simeq sXY
\end{align}
to be proved uniformly in $-1<s<1, 0<Y<X\le 1$. 

We use
$$
\frac{\sinh u}u=\sum_{k=0}^\infty \frac{u^{2k}}{(2k+1)\,!}, \qquad u>0,
$$
and expand each of the four terms in $S=S(s,X,Y)$ 
to obtain
\begin{equation*}
S=\sum_{j=0}^\infty \frac{[(1+s)(X+Y)]^{2j}}{(2j+1)\,!}\sum_{k=0}^\infty \frac{[(1-s)(X-Y)]^{2k}}{(2k+1)\,!}
-\sum_{j=0}^\infty \frac{[(1-s)(X+Y)]^{2j}}{(2j+1)\,!}\sum_{k=0}^\infty \frac{[(1+s)(X-Y)]^{2k}}{(2k+1)\,!}.
\end{equation*}
Since each of these four series is absolutely convergent, we can use Cauchy's theorem on multiplication of series to write $S(s,X,Y)=$ as the 
absolutely convergent series
$$
S(s,X,Y)=\sum_{m=1}^\infty S_m(s,X,Y):=\sum_{m=1}^\infty \Big[ \sum_{j+k=m} c_{jk}\Big(L_{jk}(s,X,Y)-L_{jk}(-s,X,Y)\Big)\Big],
$$
where $c_{jk}=\big((2j+1)\,!(2k+1)\,!\big)^{-1}$, and
$$
L_{jk}(s,X,Y)=(1+s)^{2j}(1-s)^{2k}(X+Y)^{2j} (X-Y)^{2k}.
$$

Observe at this moment that for $j+k=m$, $m\ge1$, 
$$
L_{jk}(s,X,Y)-L_{kj}(-s,X,Y)=(1+s)^{2j}(1-s)^{2k}P_{jk}^{(m)}(X,Y),
$$ 
where 
$$
P_{jk}^{(m)}(X,Y)=(X+Y)^{2j}(X-Y)^{2k}-(X+Y)^{2k} (X-Y)^{2j}
$$ 
is  a polynomial of two variables of degree $2m$, except for the case when $j=k=m/2$ with $m$ even, when this polynomial is zero. Moreover, 
$P_{kj}^{(m)}=-P_{jk}^{(m)}$, $P_{jk}^{(m)}$ does not contain the monomials $X^{2m}$ and  $Y^{2m}$,  and the monomials of degree $\le1$. 
Therefore $P_{jk}^{(m)}$ can be written as $P_{jk}^{(m)}(X,Y)=XY\widehat P_{jk}^{(m)}(X,Y)$, where $\widehat P_{jk}^{(m)}$ is  a polynomial 
of degree $2(m-1)$, and  $\widehat P_{jk}^{(m)}(X,Y)>0$ for $0<Y<X$ and $k<j$.

Consequently, since $c_{jk}=c_{kj}$,
\begin{align*}
S_m(s,X,Y)&=\sum_{j=0}^m  c_{j,m-j}\big(L_{j,m-j}(s,X,Y)-L_{m-j,j}(-s,X,Y)\big)\\
&=\sum_{j=0}^m  c_{j,m-j}(1+s)^{2j}(1-s)^{2(m-j)}P_{j,m-j}^{(m)}(X,Y)\\
&=\sum_{j=0}^{\lfloor\frac{m-1}2\rfloor}  c_{j,m-j}P_{j,m-j}^{(m)}(X,Y) \big((1+s)^{2j}(1-s)^{2(m-j)}-(1+s)^{2(m-j)}(1-s)^{2j}\big)\\
&=\sum_{j=0}^{\lfloor\frac{m-1}2\rfloor}  c_{j,m-j} P_{j,m-j}^{(m)}(X,Y)P_{j,m-j}^{(m)}(1,s)\\
&=sXY\sum_{j=0}^{\lfloor\frac{m-1}2\rfloor}  c_{j,m-j} \widehat P_{j,m-j}^{(m)}(X,Y)\widehat P_{j,m-j}^{(m)}(1,s).
\end{align*}
Thus we singled out the required factor $sXY$ and finally, for $(s,X,Y)\in\R^3$,  
$$
\sum_{m=1}^\infty S_m(s,X,Y)=sXY\sum_{m=1}^\infty\Big[\sum_{j=0}^{\lfloor\frac{m-1}2\rfloor}  c_{j,m-j} 
\widehat P_{j,m-j}^{(m)}(X,Y)\widehat P_{j,m-j}^{(m)}(1,s)\Big]=:sXY\sum_{m=1}^\infty P_m(s,X,Y).
$$
The latter series is absolutely convergent, hence 
$$
P(s,X,Y):=\sum_{m=1}^\infty P_m(s,X,Y)
$$
is indeed well-defined, in addition, with positive terms and $P_1(s,X,Y)=16$. Since $P$ is continuous on $\R^3$ and separated from zero on 
$[0,1]\times \{(X,Y)\colon 0<Y<X\}$, we have $P(s,X,Y)\simeq 1$ uniformly in $0<s<1$ and $0<Y<X\le1$, and the proof of \eqref{duo2} and thus 
also \eqref{duo} is completed. (Continuity of $P$ follows from the fact that for any $r>0$, $P_m(s,X,Y)\le p_m(r)$ on $(-r,r)^3$, with 
$\sum_1^\infty p_m(r)<\infty$; such estimate is easily seen by using crude bounds of $\widehat P_{j,m-j}^{(m)}(X,Y)$ and $\widehat P_{j,m-j}^{(m)}(1,s)$ 
from the restricted range of $X,Y,s$, and suitable estimates of $c_{j,m-j}$, see Appendix (C), where similar case is discussed.)

\noindent \textit{Proof of Claim 2}. 
We shall prove that
\begin{equation}\label{new}
 \partial_XG(s,X,Y)\simeq 1-s,
\end{equation}
uniformly in $0<s<1$, $Y>0$, and $1\vee Y<X\le Y+\frac1{1-s}$. Using \eqref{new} gives  \eqref{cl2}, since
$$
G(s,X,Y)=\int_Y^X \partial_T G(s,T,Y)\,dT\simeq (1-s)(X-Y);
$$
recall that $G(s,Y^+,Y)=0$. It should be added at this point that the positivity of $ \partial_XG(s,X,Y)$ in the considered  range of variables 
is implicitly included in the proof of \eqref{new}. More precisely, it follows from positivity of $L_X(s)$; see the remark following \eqref{inacz}.

To obtain \eqref{new}, applying \eqref{sin-cos} for the third equality below, we write 
\begin{align*}
 &\,\,\,\,\,\,\,\partial_XG(s,X,Y)\\
&= \partial_X\Big(1-\frac{\sinh(sX)}{\sinh X}\frac{\sinh Y}{\sinh(sY)}\Big)\\
&=\frac{\sinh Y}{\sinh(sY)}\frac{\sinh(sX)\cosh X-s\cosh(sX)\sinh X}{\sinh^2X}\\
&=\frac12\frac{\sinh Y}{\sinh(sY)} \frac{\sinh\big((1+s)X\big)-\sinh\big((1-s)X\big) -s\big(\sinh\big((1+s)X\big)+\sinh\big((1-s)X\big)\big)}{\sinh^2X}\\
&=\frac12 \frac{\sinh Y\sinh(sX)}{\sinh^2X\sinh(sY)} \frac{(1-s)\sinh\big((1+s)X\big) -(1+s)\sinh\big((1-s)X\big)}{\sinh(sX)}\\
&=\frac12(1-s) \frac{\sinh Y\sinh(sX)}{\sinh^2X\sinh(sY)} \frac{\sinh\big((1+s)X\big) -\frac{1+s}{1-s}\sinh\big((1-s)X\big)}{\sinh(sX)}\\
&\simeq(1-s)\frac{\sinh Y \sinh(sX)}{\sinh^2X\sinh(sY)} e^X,
\end{align*}
where in the last step we used 
\begin{equation}\label{inacz}
L_X(s):=\frac{\sinh\big((1+s)X\big) -\frac{1+s}{1-s}\sinh\big((1-s)X\big)}{\sinh(sX)}\simeq e^X, \qquad X>1,\quad 0<s<1,
\end{equation}
the  equivalence to be proved in a moment. We add that positivity of $L_X(s)$ for $0<s<1$ and $X>0$ is included in the proof of \eqref{inacz}; 
see the final inequality concluding the proof of  \eqref{inacz}.

To finish the proof of \eqref{new} it remains to verify that
\begin{equation}\label{ble}
\frac{\sinh Y \sinh(sX)}{\sinh^2X\sinh(sY)} e^X\simeq 1,
\end{equation}
uniformly in $0<s<1$, $Y>0$, and $1\vee Y<X\le Y+\frac1{1-s}$.
Applying the basic assumption $0<(1-s)(X-Y)\le1$, we have 
\begin{align*}
\frac{\sinh Y \sinh(sX)}{\sinh^2X\sinh(sY)} e^X\simeq \frac{\frac Y{Y+1}e^Y \frac {sX}{sX+1}e^{sX}}{e^{2X}\frac {sY}{sY+1}e^{sY}}e^X
&=\frac X{Y+1}\frac{sY+1}{sX+1}e^{-(1-s)(X-Y)}\\
&\simeq \frac X{Y+1}\frac{sY+1}{sX+1}\simeq \frac{sY+1}{sX+1}\times
\begin{cases}
X,& Y\le1,\\
X/Y,& Y>1.
\end{cases}
\end{align*}
For $0<Y\le 1$, using $0<sY<1$, we obtain  
$$
X\frac{sY+1}{sX+1}\simeq
	\begin{cases}
		\frac1s, &\frac1s< X,\\
		X, &X\le \frac 1s.
	\end{cases}
$$
Now, if $\frac1s< X$, then $\frac1s< X<Y+\frac1{1-s}\le1+\frac1{1-s}$, which shows that the admitted $s$ satisfies $s>s_1=(3-\sqrt{5})/2$, 
and hence $\frac1s\simeq1$ follows. Similarly, if $X\le\frac1s$, then $X\le \min(\frac1s,1+\frac1{1-s})\le2$, hence $X\simeq1$, and \eqref{ble} follows. 
For $Y>1$, we obtain
$$
\frac XY\frac{sY+1}{sX+1}\simeq	\begin{cases}
		1, &Y>\frac1s,\\
		\frac1{sY}, &Y\le\frac1s<X,\\
		\frac XY, &Y<X\le\frac1s.
	\end{cases}
$$
But in the case $Y\le\frac1s<X$ we have $s<sY\le1$, and hence, for $\frac12<s<1$, $\frac1{sY}\simeq 1$ follows. For $0<s\le\frac12$, we have 
$Y<X<Y+2$, so $Y<\frac1s<Y+2$ and hence $1<\frac1{sY}<1+\frac2y<3$, and consequently, again $\frac1{sY}\simeq 1$ follows. 
Finally, in the case $Y<X\le\frac1s$, for $\frac12<s<1$ we have $1<y<x\le 2$, so $\frac XY\simeq1$, while for $0<s\le\frac12$ we have 
$1<Y<X<Y+2$, so  $1<\frac1{sY}<1+\frac2y<3$, and again $\frac XY\simeq1$. Summing up, also in the case $Y> 1$, \eqref{ble} follows. 

In this way we finished verification of \eqref{ble} and hence the proof of  \eqref{new}. We now return to the proof of \eqref{inacz}. We first 
show that $L_X(s)$ is increasing on $(0,1)$. 

Checking that $\frac d{ds}L_X>0$ on $(0,1)$ reduces to verifying the sign of the expression 
appearing in the numerator of $\frac d{ds}L_X$, namely
\begin{align*}
n_X(s)&=\Big(\sinh\big((1+s)X\big)-\frac{1+s}{1-s}\sinh\big((1-s)X\big)\Big)'\sinh(sX)\\
&\,\,-\Big(\sinh\big((1+s)X\big)-\frac{1+s}{1-s}\sinh\big((1-s)X\big)\Big)\big(\sinh(sX)\big)'.
\end{align*}
A routine calculation shows that
$$
n_X(s)=\frac{2sX^2}{1-s}\Big(\frac{\sinh X}X-\frac{\sinh(sX)}{sX}\,\frac{\sinh\big((1-s)X\big)}{(1-s)X}\Big)
$$
and the expression in the parentheses is indeed positive due to the easy to establish inequality
\begin{equation}\label{ine}
\frac{\sinh A}{A}\,\frac{\sinh B}{B}<\frac{\sinh(A+B)}{A+B}, \qquad A,B>0.
\end{equation}

Now, an application of de L'Hospital's rule gives for $X>1$ 
\begin{equation*}
L_{X}(0^+)=\lim_{s\to 0^+}\frac{\sinh\big((1+s)X\big)-\frac{1+s}{1-s}\sinh\big((1-s)X\big)}{\sinh sX}=2\cosh X-2\frac{\sinh X}X.
\end{equation*}
On the other hand,
\begin{equation*}
L_{X}(1^-)=\lim_{s\to 1^-}=\frac{\sinh\big((1+s)X\big)-X(1+s)\frac{\sinh\big((1-s)X\big)}{(1-s)X}} {\sinh sX}= 2\cosh X-2\frac X{\sinh X}.
\end{equation*}
Putting things together shows \eqref{inacz}, namely
$$
e^X\simeq 2\cosh X\Big(1-\frac{\tanh X}X\Big) \le L_X(s)\le 2\cosh X\simeq e^X, \qquad X>1.
$$

\noindent \textit{Proof of Claim 3}. 
The upper bound (by 1) in  \eqref{cl1} is obvious. For the lower bound we shall prove that
$$
1-\frac1{\sinh1}<1-\frac{\sinh(sX)}{\sinh X}\frac{\sinh Y}{\sinh(sY)}\,, 
$$
uniformly in $0<s<1$, $Y>0$, and $X>Y+\frac1{1-s}$. 

Since for given $0<s<1$ the function $\frac{\sinh(sX)}{\sinh X}$ decreases when $X$ increases, cf. Lemma \ref{lem:bas}, hence it remains to show that
$$
G_s(y):=\frac{\sinh s\big(y+\frac1{1-s}\big)\,\sinh y}{\sinh \big(y+\frac1{1-s}\big)\,\sinh sy}<\frac1{\sinh1},\qquad 0<s<1,\quad y>0.
$$
An application of the formula $\sinh(\a+\b)=\sinh\a\cosh\b+\cosh\a\sinh\b$ leads to
$$
G_s(y)=\frac{\sinh\frac s{1-s}}{\sinh\frac 1{1-s}}\,\frac{\coth sy+\coth \frac s{1-s}}{\coth y+\coth \frac 1{1-s}}.
$$
We now observe that 
$$
G_s(y)<\frac1s\frac{\sinh\frac s{1-s}}{\sinh\frac 1{1-s}}<\frac1{\sinh1}, \qquad 0<s<1,\quad y>0.
$$
Indeed, the first inequality above (note that for $0<s<1$, $\frac1s\frac{\sinh\frac s{1-s}}{\sinh\frac 1{1-s}}=G_s(0^+)$), after a rearrangement is
$$
s\coth sy+s\coth \frac s{1-s}<\coth y+\coth \frac 1{1-s},
$$
and this inequality is an immediate consequence of the fact that $s\coth(sA)<\coth A$ for $A>0$, which is equivalent to the statement that  $\frac{\tanh u}u$ 
is decreasing on $(0,\infty)$. On the other hand, the second required inequality (notice that $\frac1{\sinh1}=\lim_{s\to0^+}G_s(0^+)$) is a consequence of 
$$
\frac{\sinh1}1\frac{\sinh B}{B}<\frac{\sinh(1+B)}{1+B},
$$
taken with $B=\frac s{1-s}$,  which follows from \eqref{ine}. This also finishes the proof of the three Claims, hence the proof of \eqref{ort}, and thus Theorem \ref{thm:2}.

\begin{remark}\label{long2} 
It follows from \eqref{Neu} and Propositions \ref{pro:N2}, \ref{pro:N1}, and Theorem \ref{thm:2}, what the long-time decay of $p_t^{\eta, C_+^{(4)}}$ 
is for $\eta={\rm triv},\eta_{\calN_1},\eta_{\calN_1},\sgn$. Namely, given $x,y\in C_+^{(4)}$ we have for $t$ large,  
$p_t^{{\rm triv},C_+^{(4)}}(x,y)\simeq t^{-1}$, $p_t^{\eta_{\calN_1},C_+^{(4)}}(x,y)\simeq t^{-3}\simeq  p_t^{\eta_{\calN_2},C_+^{(4)}}(x,y)$, 
$p_t^{\sgn,C_+^{(4)}}(x,y)\simeq t^{-5}$. Again the number of facets of the Weyl chamber $C_+^{(4)}$ with imposed Dirichlet type boundary conditions matters. 
\end{remark}

\section{Dihedral root system $I_2(3)$ case} \label{sec:dih2}
For an introduction to this section see \cite[Section 5.3]{S2}. We only recall that the open cone on the plane 
$$
C_+^{(3)}=\big\{\rho e^{i \theta}\colon \rho>0,\,\,\,\,|\theta|<\pi/6\}=\{(x_1,x_2)\colon |x_2|<\frac1{\sqrt3}x_1\big\},
$$
with vertex at the origin and aperture $\pi/3$, is the distinguished Weyl chamber in this context. 

There are only two homomorphisms of the corresponding Weyl group, \textsl{triv} and \textsl{\sgn}, and the  Dirichlet heat kernel on $C_+^{(3)}$ is 
$$
p_t^{{\rm det},C_+^{(3)}}(x,y)=e^{-\langle x,y\rangle /2t}\Phi_{t,\sgn}(x,y)p_t^{(2)}(x-y),
$$
where
\begin{align*}
\Phi_{t,\sgn}(x,y)=\exp\big(\frac{x_2y_2}{2t}\big)\sinh\big(\frac{x_1y_1}{2t}\big)&-\exp\Big(\frac{(\sqrt{3}x_1-x_2)y_2}{4t}\Big)\sinh\Big(\frac{(x_1+\sqrt{3}x_2)y_1}{4t}\Big)\\
&-\exp\Big(\frac{(-\sqrt{3}x_1-x_2)y_2}{4t}\Big)\sinh\Big(\frac{(x_1-\sqrt{3}x_2)y_1}{4t}\Big). 
\end{align*}
Recall that by the general theory, the  Dirichlet heat kernel $p_t^{{\rm det},C_+^{(3)}}(x,y)$  is positive on $C_+^{(3)}$; equivalently, $\Phi_{t,\sgn}(x,y)$ 
is positive on $C_+^{(3)}$. This means that the left-hand sides in \eqref{est} and \eqref{est3} below, are positive on suitable regions. It is worth noting  
that positivity of the expression on the left-hand side of \eqref{est3}, and thus positivity of $p_t^{{\rm det},C_+^{(3)}}$, was proved by a direct 
computation in \cite[Section 5.3]{S2}. Moreover, $p_t^{{\rm det},C_+^{(3)}}(\cdot,y)$ vanishes on $\partial C_+^{(3)}$ for every $y\in C_+^{(3)}$. 

We claim that 
$$
p_t^{{\rm det},C_+^{(3)}}(x,y)\simeq  \frac{x_1y_1}{x_1y_1+t}\, \frac{(x_1-\sqrt3 x_2)(y_1-\sqrt3 y_2)}{(x_1-\sqrt3 x_2)(y_1-\sqrt3 y_2)+t}
\, \frac{(x_1+\sqrt3 x_2)(y_1+\sqrt3 y_2)}{(x_1+\sqrt3 x_2)(y_1+\sqrt3 y_2)+t}\,p_t^{(2)}(x-y),
$$
uniformly in $x=(x_1,x_2)\in C_+^{(3)},\,y=(y_1,y_2)\in C_+^{(3)}$, and  $t>0$. This amounts to showing that 
$$
e^{-\langle x,y\rangle /2t}\Phi_{t,\sgn}(x,y)
$$
is  equivalent to the product of the three terms preceding $p_t^{(2)}(x-y)$ in the claimed equivalence. Clearly, we can consider only $t=1/2$. Then
\begin{align*}
e^{-\langle x,y\rangle}&\Phi_{1/2,\sgn}(x,y)=e^{-\langle x,y\rangle} e^{x_2y_2}\sinh x_1y_1\times\\
&\Big(1-\frac{\exp\big(\frac{(\sqrt{3}x_1-x_2)y_2}{2}\big)}{e^{x_2y_2}}  \frac{\sinh\big(\frac{(x_1+\sqrt{3}x_2)y_1}{2}\big)}{\sinh x_1y_1}
-\frac{\exp\big(\frac{(-\sqrt{3}x_1-x_2)y_2}{2}\big)}{e^{x_2y_2}}  \frac{\sinh\big(\frac{(x_1-\sqrt{3}x_2)y_1}{2}\big)}{\sinh x_1y_1}\Big).
\end{align*}
But
$$
e^{-\langle x,y\rangle} e^{x_2y_2}\sinh x_1y_1=\frac12(1-e^{-2x_1y_1})\simeq  \frac{x_1y_1}{x_1y_1+1}
$$
hence one required term is singled out and we are reduced to showing that,  uniformly in $x,y\in C_+^{(3)}$, 
\begin{align}\label{est}
1-&\frac{ \exp\big(\frac{(\sqrt 3x_1-3x_2)y_2}2\big)\sinh\big(\frac{(x_1+\sqrt{3}x_2)y_1}2\big) +
\exp\big(-\frac{(\sqrt 3x_1+3x_2)y_2}2\big)\sinh\big(\frac{(x_1-\sqrt{3}x_2)y_1}2\big)}{\sinh x_1y_1}\simeq \,\,\,\,\,\,\,\,\,\,\,\,\nonumber\\
&\,\,\,\,\,\,\,\,\,\,\,\,\,\,\,\,\,\,\,\,\,\,\,\,\frac{(x_1-\sqrt3 x_2)(y_1-\sqrt3 y_2)}{(x_1-\sqrt3 x_2)(y_1-\sqrt3 y_2)+1}\,\frac{(x_1+\sqrt3 x_2)(y_1+\sqrt3 y_2)}{(x_1+\sqrt3 x_2)(y_1+\sqrt3 y_2)+1}.
\end{align}
For $x,y\in C_+^{(3)}$ we will use a modification of variables introduced in the preceding section,
$$
X=\frac{(x_1-\sqrt 3x_2)y_1}2, \quad Y=\frac{(x_1+\sqrt 3x_2)y_1}2, \quad s=\sqrt3 \frac{y_2}{y_1},
$$ 
so that $|s|<1$, $X>0$, $Y>0$. In these new variables, \eqref{est} is equivalent to 
\begin{equation}\label{est3}
 G(s,X,Y):=1-\frac{\exp(sX)\sinh Y+ \exp(-sY)\sinh X}{\sinh(X+Y)}\simeq \frac{(1-s)X}{(1-s)X+1}\frac{(1+s)Y}{(1+s)Y+1},
\end{equation}
uniformly in $X,Y>0$, $|s|<1$. Recall that $G$ is positive in the truncated layer 
$$
\calL_{\rm tr}=(-1,1)\times(0,\infty)\times(0,\infty),
$$
and so is the expression in \eqref{lala2} below, denoted $S(s,X,Y)$  in the beginning of Section \ref{ssec:first}. Moreover, $G$ vanishes on the boundary of 
$\calL_{\rm tr}$, which is equivalent to the fact that  $\Phi_{1/2,\sgn}(x,y)$ vanishes if either $x$ or $y$ belong to the boundary of $C_+^{(3)}$. 

To prove \eqref{est3} we shall consider the main cases, when $X+Y<1$ or $X+Y\ge1$, and four subcases in the latter case; this will be done in two 
subsequent sections. If $X+Y<1$, then also $(1-s)X<2$ and $(1+s)Y<2$ and,  due to $\sinh(X+Y)\simeq X+Y$, \eqref{est3} becomes
\begin{equation}\label{lala2}
\sinh(X+Y) -\big(\exp(sX)\sinh Y+ \exp(-sY)\sinh X\big) \simeq (1-s^2)XY(X+Y).
\end{equation}
If $X+Y\ge1$, then \eqref{est3} becomes
\begin{equation}\label{lala5}
G(s,X,Y) \simeq	\begin{cases}
  (1-s^2)XY, &(1-s)X<1,\,(1+s)Y<1,\,\,\,\, (A)\\
		(1-s)X, &(1-s)X<1,\,(1+s)Y\ge1,\,\,\,\, (B)\\
		(1+s)Y, &(1-s)X\ge1,\,(1+s)Y<1,\,\,\,\, (C)\\
		1, &(1-s)X\ge1,\,(1+s)Y\ge1.\,\,\,\, (D)
	\end{cases}
\end{equation} 
Notice that due to $G(s,Y,X)=G(-s,X,Y)$ the second and third cases are symmetric. 

In what follows, to avoid a multiplication of symbols, in the following two subsections we shall use the same characters (some of them already used in 
Section \ref{sec:dih}), like $S$, $P$, $\calR$, and their variants, to denote different object but related to similar contexts. For instance, $\calR$ will 
always stand for a region defined by conditions imposed on $s$, $X$, and $Y$.

\subsection{The  case $X+Y<1$} \label{ssec:first}
We commence proving \eqref{lala2} with the leading assumption $X+Y<1$. For $s,X,Y$ real, denoting by $S(s,X,Y)$ the left-hand side of \eqref{lala2}, we have 
\begin{equation*}
S(s,X,Y)=\sum_{k=0}^\infty\frac{(X+Y)^{2k+1}}{(2k+1)!}-\Big[\sum_{k=0}^\infty\frac{(sX)^{k}}{k!}\sum_{k=0}^\infty\frac{Y^{2k+1}}{(2k+1)!} + 
\sum_{k=0}^\infty\frac{(-sY)^{k}}{k!}\sum_{k=0}^\infty\frac{X^{2k+1}}{(2k+1)!} \Big].
\end{equation*}
We will show that 
\begin{equation}\label{AA2}
S(s,X,Y)=(1-s^2)XY(X+Y)P(s,X,Y),\qquad (s,X,Y)\in\R^3,
\end{equation}
where 
\begin{equation}\label{AA3}
P(s,X,Y)\simeq 1, \qquad (s,X,Y)\in\calR,
\end{equation}
and
$$
\calR=(-1,1)\times\{(X,Y)\colon X>0,Y>0,X+Y<1\}.
$$ 

Expanding the summands $(X+Y)^{2k+1}$ and using the Cauchy theorem mentioned in Section \ref{sec:dih}, we rearrange the resulting summands and obtain
$$
S(s,X,Y)=\sum_{m=3}^\infty S_m(s,X,Y), 
$$
where, for any $m\ge3$, $S_m(s,X,Y)$ is a sum of monomials in the variables $X$ and $Y$ of joint degree $m$, and coefficients depending 
polynomially on $s$. Moreover, the last series  converges  absolutely. Notice that $S$ does not include a constant 
term, and monomials of joint degree 1 and 2 do not appear due to an easily seen cancellation. First, we aim to show that for every $m\ge3$,
\begin{equation}\label{qq}
S_m(s,X,Y)=(1-s^2)XY(X+Y)P_{m}(s,X,Y), \qquad (s,X,Y)\in\R^3,
\end{equation}
where, for any fixed $s$, $P_{m}(s,\cdot,\cdot)$ is a polynomial of two variables $X$ and $Y$ of joint degree $m-3$. We now begin the analysis of $S_m$ 
splitting the consideration onto odd and even cases.

Suppose first that $m\ge3$ is odd. Then $S_m$ incorporates monomials in $X$ and $Y$ of joint degree $m$ of two kinds. Firstly, we have the terms coming from 
the expansion of $(X+Y)^{m}$, $m=2k+1$, except for the monomials $X^m$ and $Y^m$ which are canceled (see below). Therefore, the sum of monomials of the first kind is    
\begin{equation*}
S_m^{(1)}=\sum_{j=1}^{m-1}\frac{X^jY^{m-j}}{j!(m-j)!}
=XY\sum_{j=0}^{m-2}\frac{X^jY^{m-j-2}}{(j+1)!(m-j-1)!}=:XYJ_m^{(1)}.
\end{equation*}
Secondly, included are the monomials that come from  multiplication of  terms with even exponents from the series representing $\exp(sX)$ or $\exp(-sY)$, 
and a term in the series representing $\sinh Y$ or $\sinh X$, respectively, except for the monomials $X^m$ and $Y^m$ which are used to the just mentioned cancellation. Therefore, not including at this moment the minus sign, the sum of monomials of the second  kind is    
\begin{align*}
S_m^{(2)}&=\sum_{j=1}^{\frac{m-1}2}\Big[\frac{(sX)^{2j}}{(2j)!} \frac{Y^{m-2j}}{(m-2j)!} + \frac{(-sY)^{2j}}{(2j)!} \frac{X^{m-2j}}{(m-2j)!}\Big]\\
&=XY\sum_{j=0}^{\frac{m-3}2}\frac{s^{2(j+1)}}{(2(j+1))!(m-2j-2)!}\big(X^{2j+1}Y^{m-3-2j}+Y^{2j+1}X^{m-3-2j}\big)=:XYJ_m^{(2)}.
\end{align*}
Putting all these monomials together gives  $S_m=S_m^{(1)}-S_m^{(2)}=XY\big(J_m^{(1)}-J_m^{(2)}\big)$; thus, $XY$ is singled out. Now, in the sum defining 
$J_m^{(1)}$ we pair the terms with indices $2j$ and $m-2-2j$, $j=0,\ldots,\frac{m-3}2$, and subtract the $(\frac{m-3}2-j)$-th term from $J_m^{(2)}$ to obtain
\begin{align*}
J_m^{(1)}-J_m^{(2)}&=\sum_{j=0}^{\frac{m-3}2}\frac{1-(s^2)^{\frac{m-1}2-j}}{(2j+1)!(m-2j-1)!}\Big(X^{2j}Y^{m-2-2j}+Y^{2j}X^{m-2-2j} \Big)\\
&=(1-s^2)(X+Y)\sum_{j=0}^{\frac{m-3}2}\frac{Q_{m,j}(s)R_{m,j}(X,Y)}{(2j+1)!(m-2j-1)!},
\end{align*}
where (notice that $\frac{m-1}2-j\ge1$ for $0\le j\le \frac{m-3}2$)
$$
Q_{m,j}(s):=\frac{1-(s^2)^{\frac{m-1}2-j}}{1-s^2}, \qquad R_{m,j}(X,Y)=\frac{X^{2j}Y^{m-2-2j}+Y^{2j}X^{m-2-2j}}{X+Y}. 
$$
Letting $Q_{m,j}(\pm1)=\frac{m-1}2-j$ makes $Q_{m,j}$ continuous on $\R$; $Q_{m,j}$ is then an even polynomial of degree $m-3-2j$, positive on $\R$. 
Set $r(j,m):=\min(2j, m-2-2j)$ so that $r(0,m)=0$ and $r(j,m)\ge1$ for $j=1,\ldots, \frac{m-3}2$. Observing that $m-2-2r(j,m)\ge1$ is odd shows that 
$R_{m,j}(X,Y)$ equals
$$
(XY)^{r(j,m)}\frac{X^{m-2-2r(j,m)}+Y^{m-2-2r(j,m)}}{X+Y}=(XY)^{r(j,m)}\sum_{i=0}^{m-3-2r(j,m)}(-1)^{i}X^{m-3-2r(j,m)-i}Y^i. 
$$
Concluding, for $m$ odd $S_m$  is of the form \eqref{qq} with 
\begin{equation}\label{pmo}
P_m(s,X,Y)=\sum_{j=0}^{\frac{m-3}2}\frac{Q_{m,j}(s)R_{m,j}(X,Y)}{(2j+1)!(m-2j-1)!}.
\end{equation}
Notice that $R_{m,j}$ are polynomials in $X$ and $Y$, positive for $X,Y>0$, and of degree $m-3$, hence  $P_m(s,X,Y)$ is a polynomial in $X$ and $Y$ 
of degree $m-3$ with coefficients depending on $s$, positive for $s\in\R$ and $X,Y>0$; for instance $P_3\equiv\frac12$. 

Assume now that $m\ge4$ is even. Then $S_m$ incorporates all monomials in $X$ and $Y$ of joint degree $m$ that come from multiplication of a term with odd 
exponent from the series representing $\exp(sX)$ and a term in the series representing $\sinh Y$, or a term with odd exponent from the series 
representing $\exp(-sY)$ and a term in the series representing $\sinh X$,  respectively. Thus, 
\begin{align*}
S_m&=-\sum_{j=0}^{\frac m2-1}\Big[\frac{(sX)^{2j+1}}{(2j+1)!}\frac{Y^{m-2j-1}}{(m-2j-1)!} + \frac{(-sY)^{2j+1}}{(2j+1)!}\frac{X^{m-2j-1}}{(m-2j-1)!} \Big]\\
&=XY\sum_{j=0}^{\frac{m-2}2} \frac{s^{2j+1}}{(2j+1)!(m-2j-1)!}\Big[Y^{2j}X^{m-2j-2} - X^{2j}Y^{m-2j-2} \Big]=:XYI_m.
\end{align*}
Now, if $\frac m2$ is odd, then the middle term in the above summation, i.e. the term corresponding to $j=\frac{m-2}4$, vanishes. 
Therefore, whatever $\frac m2$ is, pairing in $I_m$ the terms with indices $j$ and $\frac{m-2}2-j$ gives
\begin{align*}
I_m&=s\sum_{j=0}^{\lfloor \frac m4\rfloor-1} \frac{s^{2j}}{(2j+1)!(m-2j-1)!}\big(1-(s^2)^{\frac{m-2}2-2j}\big)\Big[Y^{2j}X^{m-2j-2} - X^{2j}Y^{m-2j-2} \Big]\\
&=(1-s^2) s\sum_{j=0}^{\lfloor \frac m4\rfloor-1} \frac{Q_{m,j}(s)}{(2j+1)!(m-2j-1)!}(XY)^{2j}\big(X^{m-4j-2} - Y^{m-4j-2}\big)\\
&=(1-s^2) (X+Y)(X-Y)s\sum_{j=0}^{\lfloor \frac m4\rfloor-1} \frac{Q_{m,j}(s)R_{m,j}(X,Y)}{(2j+1)!(m-2j-1)!},
\end{align*}
where (notice that $\frac{m}2-2j-1\ge1$ for $0\le j\le \lfloor\frac{m}4 \rfloor-1$)
$$
Q_{m,j}(s):=s^{2j}\frac{1-(s^2)^{\frac m2-2j-1}}{1-s^2}
$$
is, after appropriate adjustment at $\pm1$, an even positive on $\R$ polynomial of degree $m-4-4j$,  and
$$
R_{m,j}(X,Y):=(XY)^{2j}\frac{(X^2)^{\frac m2-2j-1} - (Y^2)^{\frac m2-2j-1}}{X^2-Y^2}=(XY)^{2j}\sum_{i=0}^{\frac m2-2j-2}(X^2)^{\frac m2-2j-2-i}(Y^2)^{i}.
$$
Thus, for $m$ even, $S_m=XYI_m$ is of the form \eqref{qq} with
\begin{equation}\label{pme}
P_m(s,X,Y)= s(X-Y)\sum_{j=0}^{\lfloor \frac m4\rfloor-1} \frac{Q_{m,j}(s)R_{m,j}(X,Y)}{(2j+1)!(m-2j-1)!}.
\end{equation}
Since  $Q_{m,j}$ are positive  and $R_{m,j}$ are positive for $X,Y\in\R$, the last sum is positive for $s\in\R$ and $X,Y\in\R$, and thus the sign of 
$P_m(s,X,Y)$ is determined by the sign of $s(X-Y)$. For instance $P_4\equiv\frac16s(X-Y)$.

Finally, we reached \eqref{AA2} with 
$$
P(s,X,Y):=\sum_{m=3}^\infty P_m(s,X,Y),  \qquad (s,X,Y)\in\R^3,
$$
and $P_m$ given by \eqref{pmo} and \eqref{pme}. To complete the proof of \eqref{lala2}, and finish the proof of \eqref{est} in the case $X+Y<1$, it remains to check \eqref{AA3}; recall that $P$ is positive on $\calL_{\rm tr}$ because $S$ is. 

The series $\sum_{m=3}^\infty P_m(s,X,Y)$ is pointwise absolutely convergent (because $\sum_3^\infty S_m$ is) and its sum  is well-defined. 
For continuity of $P$ it suffices to check that the series is uniformly convergent on every cube $(-r,r)^3$, $r\to\infty$. This can be done by a direct calculation; see  Appendix (C) for details.

With continuity of $P$ on $\R^3$, we now argue as follows. On the compact set $\overline{\calR}$, the closure of $\calR$,
$P$ is bounded from above which proves the required upper bound of $P$ on $\calR$. For the lower bound we shall use an argument of a topological flavor. Namely, 
since $P(s,X,Y)$ is positive on $\calR$, it suffices to verify that $P$ does not vanish on the boundary of $\calR$. Continuity then implies 
that $P$ is separated from zero on $\calR$. For the proof that $P\neq0$, more precisely $P>0$, on $\partial \calR$ we split the argument according to the suitable partition of the boundary and check (slightly more than necessary) that:
\begin{enumerate}
	\item $P(s,0,0)>0$ for $-1\le s\le 1$;
	\item $P(\pm1,X,0)>0$ for $0<X\le1$, or $P(\pm1,0,Y)>0$ for $0<Y\le1$;
	\item $P(s,X,0)>0$ for $-1<s<1$ and $X>0$, or $P(s,0,Y)>0$ for $-1<s<1$ and $Y>0$;
	\item $P(\pm1,X,Y)>0$ for $X>0$ and $Y>0$.
\end{enumerate}
Notice that the conclusions for the second parts of Items (2) and (3) follow from the conclusions for the first parts due to the equality $P(s,Y,X)=P(-s,X,Y)$, hence we shall discuss only the first parts. Observe also that we do not discuss the part of the boundary $\partial\calR$ being the intersection of the 
plane $X+Y=1$ in $\R^3$ with $\calL_{\rm tr}$, because $P>0$ there.

For the proofs of Items (1) and (2) we need a formula on $P_m(s,X,0)$ that follows from \eqref{pmo} and \eqref{pme} and accompanying formulas for 
$R_{m,j}$ (note that $R_{m,j}(X,0)=0$ for $j\ge1$). Namely, for $s,X\in\R$, 
$$
P_m(s,X,0)=\frac{Q_{m,0}(s)}{(m-1)!}\times 
	\begin{cases}
		1, & m\ge3 \,\,{\rm odd},\\
		sX^{m-3}, & m\ge4 \,\,{\rm even}.
	\end{cases}
$$

\noindent \textbf{Item (1)}. Using the above, the positivity of $Q_{m,j}$ gives for $-1\le s\le 1$
$$
P(s,0,0)=\sum_{m=3}^\infty P_m(s,0,0)=\sum_{k=1}^\infty \frac{Q_{2k+1,0}(s)}{(2k)!}>0.
$$

\noindent \textbf{Item (2)}. This time, since that $Q_{2k+1,0}(\pm1)=k$ and $Q_{2k,0}(\pm1)=k-1$, we have for $0<X\le1$
\begin{align*}
P(\pm1,X,0)&=\sum_{m=3}^\infty P_m(\pm1,X,0)=\sum_{k=1}^\infty \frac{Q_{2k+1,0}(\pm1)}{(2k)!}\pm\sum_{k=2}^\infty \frac{Q_{2k,0}(\pm1)X^{2k-3}}{(2k-1)!}\\
&=\sum_{k=1}^\infty \frac{k}{(2k)!}\pm\sum_{k=2}^\infty\frac{(k-1)X^{2k-3}}{(2k-1)!}>\frac e2-\sum_{k=1}^\infty\frac{k}{(2k+1)!}
=\frac e2-\frac12\big(1-\frac1e\big)=\cosh1-\frac12>0.
\end{align*}

For the proofs of Items (3) and (4) our strategy is as follows. Let us focus on the first part of Item (3). 
The function $S$ is real-analytic at every point $(s,X,Y)\in\R^3$, so with $-1<s<1$ and $X>0$ fixed, $S(s,X,\cdot)$ is a real-analytic function 
of one variable with  $Y=0$ as a root. Since $S(s,X,Y)=XY(X+Y)(1-s^2)P(s,X,Y)$,  checking that $\partial_Y S(s,X,Y)_{Y=0}\neq0$ will imply that  
$Y=0$ is a simple root of $S(s,X,\cdot)$ and thus $P(s,X,0)\neq0$. Positivity of $P$ in $\calR$ actually shows that $P(s,X,0)>0$.

\noindent \textbf{Item (3)}. Fix $s\in(-1,1)$ and $X>0$. We shall check that $Y=0$ is a simple root of the function $\Phi(Y)=S(s,X,Y)$. 
We have $\Phi'(Y)=\cosh(X+Y)-e^{sX}\cosh Y+se^{-sY}\sinh X$, hence  $\Phi'(0)=\cosh X-e^{sX}+s\sinh X$. We claim that $\Phi'(0)>0$. 
Let $f(s)=\cosh X-e^{sX}+s\sinh X$, with $X>0$ fixed and $s$ freed, so that $f'(s)=-Xe^{sX}+\sinh X>0$ if and only if $\frac{\sinh X}X>e^{sX}$. 
Hence, $f(s)$ increases on $(-1,s_0)$ and decreases on $(s_0,1)$, where $s_0\in (0,1)$ is such that $\frac{\sinh X}X=e^{s_0X}$. This together 
with $f(\pm1)=0$ proves the claim.

\noindent \textbf{Item (4)}. Fix $X>0$ and $Y>0$. We shall check that $s=\pm1$ is a simple root of the function $\Psi(s):=S(s,X,Y)$. 
More precisely, we shall check that $\Psi'(\pm1)\neq0$. Indeed, $\Psi'(s)=-Xe^{sX}\sinh Y+Ye^{-sY}\sinh X$ and 
$\Psi'(1)=-Xe^{X}\sinh Y+Ye^{-Y}\sinh X<0$ since $\frac{\sinh X}{Xe^X}<\frac{e^Y\sinh Y}{Y}$ for $X,Y>0$ ($\frac{\sinh X}{Xe^X}$ decreases 
on $(0,\infty)$, while $\frac{e^Y\sinh Y}{Y}$ increases on $(0,\infty)$). Analogous argument shows that $\Psi'(-1)>0$.
 
\subsection{The case $X+Y\ge1$} \label{ssec:second}

We now continue and prove \eqref{lala5} with the leading assumption $X+Y\ge1$. 

\vskip 0.4cm
\noindent \textbf{Case  (A)}: $(1-s)X<1,\,(1+s)Y<1$. Since now $\sinh(X+Y)\simeq \exp(X+Y)$, it follows that \eqref{lala5} (A)  is equivalent to 
$$
\frac{\sinh(X+Y)}{e^{X+Y}} - \Big[e^{-(1-s)X}\frac{\sinh Y}{e^Y}+ e^{-(1+s)Y}\frac{\sinh X}{e^X}\Big]   \simeq (1-s^2)XY.
$$
But 
$$
\frac{\sinh \a}{\exp \a}=\frac12\big(1-e^{-2\a}\big), 
$$
so  neglecting the factor $\frac12$ and with the change of variables $X\to X/2$, $Y\to Y/2$, our goal is
\begin{equation}\label{A1}
S(s,X,Y):=1-e^{-(X+Y)}-e^{-\frac{1-s}2 X}\big(1-e^{-Y}\big)-e^{-\frac{1+s}2 Y}\big(1-e^{-X}\big)\simeq (1-s^2)XY,
\end{equation}
uniformly in $s,X,Y$ from the region 
$$
\calR=\Big\{(s,X,Y)\colon |s|<1,\, X+Y\ge2,\, \frac{1-s}2X<1,\,\frac{1+s}2Y<1\Big\}.
$$ 
To achieve the goal we shall apply the method already used in the case $X+Y<1$. It is worth adding at this point that $P$ is continuous and the arguments 
for this are analogous to these establishing continuity of $P$ considered in Section \ref{ssec:first}.

Observe that $S$ is a real-analytic function on $\R^3$, positive in the truncated layer $\calL_{\rm tr}$
and vanishing on the boundary of $\calL_{\rm tr}$, and we have
\begin{equation*}
S(s,X,Y)=
-\sum_{m=1}^\infty \frac{(-(X+Y))^m}{m!}+ \Big[\sum_{n=0}^\infty\frac{\big(-\frac{1-s}2X \big)^n}{n!}  \sum_{n=1}^\infty  \frac{(-Y)^n}{n!}
+\sum_{n=0}^\infty\frac{\big(-\frac{1+s}2 Y \big)^n}{n!}  \sum_{n=1}^\infty  \frac{(-X)^n}{n!}\Big]. 
\end{equation*}
We will show, and this is sufficient to obtain  \eqref{A1}, that 
\begin{equation}\label{A2}
S(s,X,Y)=(1-s^2)XY P(s,X,Y),\qquad (s,X,Y)\in\R^3,
\end{equation}
where 
\begin{equation}\label{A3}
P(s,X,Y)\simeq 1, \qquad (s,X,Y)\in\calR.
\end{equation}

Expanding the summands $(-(X+Y))^m$ and using the Cauchy theorem mentioned before, we rearrange the resulting summands and obtain
$$
S=\sum_{m=3}^\infty S_m, 
$$
where, for any fixed $s$, $S_m=S_m(s,X,Y)$ is a sum of monomials in the variables $X$ and $Y$ of joint degree $m$, and coefficients depending on $s$, with 
the last series being absolutely convergent. Notice that $S$ does not include a constant term, and monomials of joint degree 1 or 2 do not appear due to an easily seen cancellation. 

For the analysis of $S_m$, $m\ge3$, we first note that 
\begin{align*}
S_m&=\sum_{j=0}^{m-1}\frac{1}{j!(m-j)!}\Big[\big(-\frac{1-s}2X\big)^j(-Y)^{m-j}+\big(-\frac{1+s}2Y\big)^j(-X)^{m-j}  
\Big]-(-1)^m\sum_{j=0}^{m}\frac{X^jY^{m-j}}{ j!(m-j)!}\\
&=(-1)^m\sum_{j=1}^{m-1}\frac{1}{j!(m-j)!} \Big[ \big(\frac{1-s}2\big)^jX^jY^{m-j}+  \big(\frac{1+s}2\big)^jY^jX^{m-j}-X^jY^{m-j} \Big].
\end{align*}
Two summands from the first sum above corresponding to $j=0$, i.e. $\frac1{m!}(-Y)^m$ and $\frac1{m!}(-X)^m$, are canceled by the summands from 
the second sum that correspond to $j=0$ and $j=m$. 

We now continue rearranging the latter sum. For this, assuming for a moment that $m$ is odd, we pair the terms in brackets 
and corresponding to the indices $j$ and $m-j$, $0\le j\le\frac{m-1}2$, to obtain 
\begin{equation*}
\Big(\big(\frac{1-s}2\big)^j + \big(\frac{1+s}2\big)^j-1\Big)X^jY^{m-j}+ \Big(\big(\frac{1+s}2\big)^j + \big(\frac{1-s}2\big)^j-1\Big)Y^jX^{m-j}.
\end{equation*}
With the notation
$$
\widehat Q_{m,j}(s)=\Big(\frac{1-s}2\Big)^j+\Big(\frac{1+s}2\Big)^{m-j}-1,
$$
we obtain for $m\ge3$ odd
$$
S_m=(-1)^m\sum_{j=1}^{\frac{m-1}2}\frac{(XY)^j }{j!(m-j)!} \Big(\widehat Q_{m,j}(s)Y^{m-2j}+ \widehat Q_{m,j}(-s)X^{m-2j}  \Big).
$$
If $m\ge4$ is even, then the above formula remains essentially the same with the following modification: the summation now goes to $\frac m2$ and the sum 
of the last terms in the parentheses, corresponding to $j=\frac m2$, is the one-half of the previous one, i.e.  takes the form $\widehat Q_{m,\frac m2}(s)$ 
(in place of $ \widehat Q_{m,\frac m2}(s)+ \widehat Q_{m,\frac m2}(-s)$; notice that  $ \widehat Q_{m,\frac m2}(s)=\widehat Q_{m,\frac m2}(-s)$).

For $m\ge3$ and $j=1,2,\ldots, m-1$, $s=\pm1$ are roots of the polynomial $\widehat Q_{m,j}$, hence we define
$$
Q_{m,j}(s):=(-1)^m 4\frac{\widehat Q_{m,j}(s)}{1-s^2},
$$
which is a polynomial of degree $(j-2)\vee(m-j-2)$, and $Q_{m,j}(s)=Q_{m,m-j}(-s)$. A closer look at the structure of $Q_{m,j}$ reveals that 
$$
Q_{m,j}(s)=(-1)^{m+1}2^{m}\Big(\sum_{k=0}^{m-j-2}\big(\frac{1+s}2\big)^k+ \sum_{k=0}^{j-2}\big(\frac{1-s}2\big)^k \Big);
$$
here, if $m-j-2<0$ or $j-2<0$, then the corresponding sum in not taken into account. Thus, for instance, $Q_{3,1}=Q_{3,2}=8$, $Q_{4,1}=-8(3+s)$, $Q_{4,2}=-32$, 
$Q_{4,3}=-8(3-s)$.

With this notation we further rearrange $S_m$ to the form
$$
S_m=XY(1-s^2)P_m(s,X,Y),  
$$
where
$$
P_m(s,X,Y)=\sum_{j=1}^{\lfloor\frac{m}2\rfloor}\frac{Q_{m,j}(s)X^{j-1}Y^{m-j-1}+Q_{m,j}(-s)X^{m-j-1}Y^{j-1}}{j!(m-j)!},
$$
again with the convention that for $m$ even the latter sum in the numerator of the two terms  corresponding to $j=\frac m2$, is the one-half 
of the original one, i.e. equals $Q_{m,\frac m2}(s)(XY)^{m/2-1}$. For instance, $P_3=4(X+Y)$. Hence we proved  \eqref{A2} with 
$P(s,X,Y):=\sum_{m=3}^\infty P_m(s,X,Y)$ and now it remains to justify \eqref{A3}.

For this we shall use still another change of variable, i.e. $T=\frac{1-s}2X$ and $U=\frac{1+s}2Y$, which will lead to a compact set framework. More precisely, we consider the mapping $\Phi$ defined on the layer $\calL=(-1,1)\times\R\times\R$, 
$$
\Phi(s,T,U)=\Big(s,\frac{2T}{1-s}, \frac{2U}{1+s}\Big), \qquad (s,T,U)\in\calL.
$$
Obviously, $\Phi$ is a continuous bijection of $\Omega^*:=(-1,1)\times[0,\infty)\times[0,\infty)$ onto itself ($\Omega^*$ enlarges $\calL_{\rm tr}$ by a part of its boundary), which preserves the mentioned part of the boundary, and the restriction of $\Phi$ to $\calL_{\rm tr}$ is a bijective diffeomorphism of $\calL_{\rm tr}$. From now on the symbol `$*$' will be related to objects connected to the $(s,T,U)$ variables. Then we define the 
mapping $S^*=S\circ \Phi$ on $\calL$. Explicitly, 
$$
S^*(s,T,U)=1-e^{-\frac{2T}{1-s}-\frac{2U}{1+s}}-e^{-T}\big(1-e^{ -\frac{2U}{1+s}} \big)-e^{-U}\big(1-e^{ -\frac{2T}{1-s}}\big).
$$

Clearly, $S^*$ is positive on $\calL_{\rm tr}$ and vanishes on this part of $\partial \Omega^*$ on which  $S^*$ is defined. A closer look at the definition 
of $S^*$ reveals that $S^*$ can be extended to a continuous function on $\overline{\Omega^*}=\overline{\calL_{\rm tr}}$ by setting
\begin{equation}\label{big2}
S^*(\pm1,T,U)=(1-e^{\mp T})(1-e^{\mp U}), \qquad T,U\ge0;
\end{equation}
we shall use the same character $S^*$ to denote this extension. Notice that on these parts of the boundary of $\Omega^*$ where $s=\pm1$, excluding edges, $S^*$ is positive. 

With the notation $P^*=P\circ\Phi$ (notice that $P^*$ is continuous on $\calL$ because $P$ is continuous on $\R^3$), \eqref{A2} implies 
\begin{equation}\label{A5}
 S^*(s,T,U)=4TU\,P^*(s,T,U), \qquad (s,T,U)\in \calL.
\end{equation}
Now, since 
$$\Phi^{-1}(\calR)=\big\{(s,T,U)\colon  |s|<1,\, T<1,\,U<1,\,\frac T{1-s}+\frac U{1+s}\ge1\big\}=:\calR^*,
$$ 
\eqref{A3} is equivalent to the statement
$$
P^*(s,T,U)\simeq1,\qquad (s,T,U)\in \calR^*.
$$
To prove this, we shall check that $P^*$ extends to a continuous function on $\overline{\calR^*}$ that does not vanish on the boundary $\partial\calR^*$. 
This will be enough for reaching the above asymptotic equivalence and thus \eqref{A3}, since we can invoke the argument already applied in the case $X+Y<1$; 
see the end of Section \ref{ssec:first}. We visualize the shape of $\calR^*$ and $\partial\calR^*$ on the following two figures.On Figure \ref{fig:1} we 
present typical sections of $\calR^*$ with planes at fixed levels $s$ (for $s=0$ the corresponding section becomes a triangle). Then on Figure \ref{fig:2} 
we approximate the shape of $\calR^*$.
\begin{figure}[ht]
\begin{tikzpicture}
[scale=0.5, axis/.style={very thick, ->, >=stealth'}]
\fill[black!20] (0,3.2) -- (0,4) -- (4,4) -- (4,0.8) -- cycle;

\draw[axis] (0,0)  -- (8,0) ;
\draw[axis] (0,0) -- (0,8) ;

\node[label={[below]$1$}] at (4,-0.3){} ;
\node[label={[left]$1$}] at (-0.3,4){} ;

\node[label={[below]$T$}] at (8,-0.3){} ;
\node[label={[left]$U$}] at (-0.15,7.5){} ;

\filldraw[black] (0,3.15) circle (3pt) ;
\node[label={[left]$1+s$}] at (-0.3,3){} ;

\filldraw[black] (0,0.8) circle (3pt) ;
\node[label={[left]$-s\frac{1+s}{1-s}$}] at (-0.3,0.7){} ;

\node[label={$-1<s<0$}] at (5,5){} ;

\draw  (0,4)--(4,4);

\draw  (4,0)--(4,4);

\draw  (0,3.2)--(4,0.8);

\draw[dotted] (0,0.8) -- (4,0.8);


\fill[black!20] (13.6,0) -- (12.7,4) -- (16,4) -- (16,0) -- cycle;

\draw[axis] (12,0)  -- (20,0) ;
\draw[axis] (12,0) -- (12,8) ;

\node[label={[below]$1$}] at (16,-0.3){} ;
\node[label={[left]$1$}] at (11.7,4){} ;

\node[label={[below]$T$}] at (20,-0.3){} ;
\node[label={[left]$U$}] at (11.85,7.5){} ;

\filldraw[black] (13.6,0) circle (3pt) ;
\node[label={[below]$1-s$}] at (13.8,-0.3){} ;

\filldraw[black] (12.7,0) circle (3pt);
\node[label={[below]$s\frac{1-s}{1+s}$}] at (12.2,-0.3){} ;

\node[label={$0<s<1$}] at (17,5){} ;

\draw  (12,4)--(16,4);

\draw  (16,0)--(16,4);

\draw  (13.6,0)--(12.7,4);

\draw[dotted] (12.7,0) -- (12.7,4);

\end{tikzpicture}
\caption{Typical sections of $\calR^*$ with planes  at fixed levels $-1<s<0$ and $0<s<1$.}

\label{fig:1}
\end{figure}
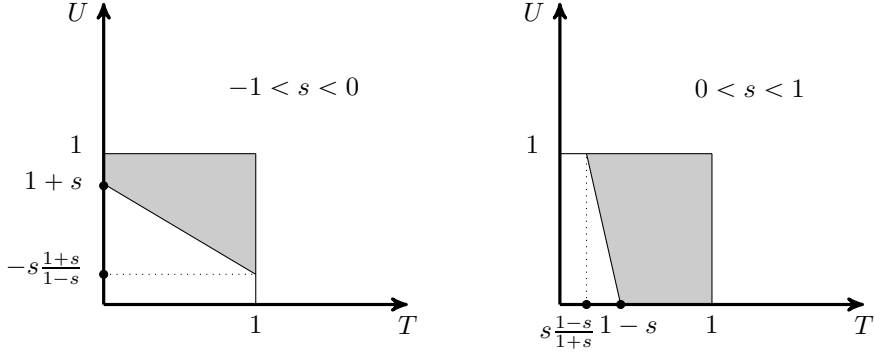

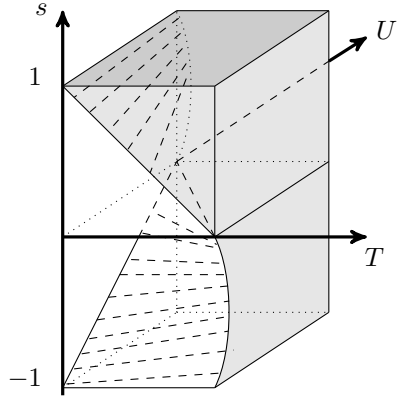
\begin{figure}[ht]
\begin{tikzpicture}
[scale=0.5, axis/.style={very thick, ->, >=stealth'}] 

\fill[black!20] (0,4) -- (4,4) -- (7,6) -- (3,6) -- cycle;
\fill[black!10] (4,0) -- (7,2) -- (7,6) -- (4,4) -- cycle;
\fill[black!10] (4,0) -- (4,4) -- (0,4) -- cycle;

\fill[black!10] (4,0) .. controls (4.5,-1) and (4.5,-3) .. (4,-4) -- (7,-2) -- (7,2)-- cycle;

\draw[axis] (0,0)  -- (8,0) ;
\draw[axis] (0,-4.2) -- (0,6) ;
\draw[axis] (7,4.66) -- (8,5.30);

\node[label={[left]$1$}] at (-0.3,4){} ;
\node[label={[left]$-1$}] at (-0.3,-4){} ;

\node[label={[below]$T$}] at (8.2,-0.3){} ;
\node[label={[left]$s$}] at (-0.15,5.8){} ;
\node[label={[right]$U$}] at (8,5.25){} ;

\draw  (0,-4)--(4,-4);
\draw  (0,4)--(4,4);
\draw  (3,6)--(7,6);
\draw[dotted]   (3,2)--(7,2);
\draw[dotted]   (3,-2)--(7,-2);

\draw  (4,0)--(4,4);
\draw  (7,-2)--(7,6);
\draw[dotted]  (3,-2)--(3,6);

\draw  (0,4)--(3,6);
\draw  (4,4)--(7,6);
\draw  (4,0)--(7,2);
\draw  (4,-4)--(7,-2);

\draw  (0,4)--(4,0);
\draw[dashed] (4,0) -- (3,2);

\draw[dotted] (0,0) -- (3,2);
\draw[dotted] (0,-4) -- (3,-2);
\draw[dashed] (3,2) -- (7,4.66);

\draw  (0,-4)--(2,0);
\draw[dashed] (2,0) -- (3,2);

\draw[dashed] (3.05, 5.8) -- (0.3,3.7);
\draw[dashed] (3.1, 5.4) -- (0.7,3.3);
\draw[dashed] (3.15, 5.0) -- (1.2,2.8);
\draw[dashed] (3.25, 4.6) -- (1.7,2.3);
\draw[dashed] (3.25, 4.2) -- (2.3,1.7);

\draw[dashed] (0.1, -3.9) -- (4.3,-3.6);
\draw[dashed] (0.3, -3.5) -- (4.5,-3.0);
\draw[dashed] (0.5, -3.1) -- (4.35,-2.5);
\draw[dashed] (0.7, -2.7) -- (4.43,-2.1);
\draw[dashed] (0.9, -2.2) -- (4.38,-1.7);%
\draw[dashed] (1.2, -1.6) -- (4.36,-1.3);
\draw[dashed] (1.5, -1.0) -- (4.28,-1.0);
\draw[dashed] (1.8, -0.6) -- (4.2,-0.7);
\draw[dashed] (2.1, 0.1) -- (4.1,-0.3);
\draw[dashed] (2.5, 0.6) -- (4.05,-0.2);

\draw (4,0) .. controls (4.5,-1) and (4.5,-3) .. (4,-4);
\draw[dotted] (3,6) .. controls (3.5,5) and (3.5,3) .. (3,2);

\end{tikzpicture}
\caption{The region $\calR^*$ in $\R^3$ with a twisted strip as a part of  $\partial\calR^*$.}

\label{fig:2}
\end{figure}

Extending $P^*$ we shall use \eqref{A5} together with the fact that $S^*$ was extended to a continuous function on $\overline{\Omega^*}$. In fact, on this part of $\partial \calR^*$, where $TU\neq0$, we simply use  \eqref{A5} to define  $P^*=S^*/4TU$. Moreover, on the considered part of $\partial \calR^*$, 
where $TU\neq0$, $P^*$ is positive since $S^*$ is positive there. It remains to consider these parts of $\partial \calR^*$, where $T=0$ or $U=0$.
We consider only the case $U=0$, since the complementary case $T=0$ is symmetric. The points of $\partial \calR^*$, where  $U=0$ are of two kinds:

\noindent Case $(a)$: points of the closed triangle with vertices at $(1,0,0)$, $(1,1,0)$,  and $(0,1,0)$;

\noindent Case $(b)$: points of the closed segment with endpoins $(-1,0,0)$  and $(-1,1,0)$.

 All the points of the triangle from Case $(a)$  with $0\le s<1$ are negligible and this claim will be explained in a moment. So we are reduced to 

\noindent Case $(a)'$: points of the closed segment with endpoins $(1,0,0)$  and $(1,1,0)$.

But then we use \eqref{big2} to define $P^*(1,T,0)$, $T>0$, as the limit
$$
P^*(1,T,0)=\lim_{U\to0^+}\frac{S^*(1,T,U)}{4TU}=\frac{1-e^{-T}}{4T},
$$
and consequently, for $T=0$,
$$
P^*(1,0,0)=\lim_{T\to0^+}\frac{P^*(1,T,0)}{4T}=\frac14.
$$
Case $(b)$ is treated in a completely analogous way. 

We now return to our last claim.  We shall justify (slightly more than necessary) that $U=0$ is a simple zero of the function $\Psi(U)=\Psi_{s,T}(U)=S^*(s,T,U)$ 
with $|s|<1$ and $T>0$ fixed. Indeed, from the explicit form of $S^*$ we have
$$
\partial_U\Psi=\frac2{1+s}e^{-\frac{2U}{1+s}}\Big(e^{-\frac{2T}{1-s}}-e^{-T} \Big)+e^{-U}\Big(  1-e^{-\frac{2T}{1-s}}\Big)
$$
and hence
$$
\partial_U\Psi(0)= \frac{1-s}{1+s}e^{-\frac{2T}{1-s}}+1-\frac2{1+s}e^{-T}:=G_s(T).
$$
But $G_s(0)=0$ and $G_s$ is easily seen to be increasing on $(0,\infty)$, so $G_s(T)>0$ for $T>0$. This finishes justification of our claim.

\vskip 0.4cm
\noindent \textbf{Case  (B)}:  $(1-s)X<1,\,(1+s)Y\ge1$; our goal is $G(s,X,Y)\simeq (1-s)X$. We shall prove that
\begin{equation*}
\partial_X G(s,X,Y)\simeq 1-s,
\end{equation*}
uniformly in $(s,X,Y)$ from the region
$$
\calR=\big\{(s,X,Y)\colon |s|<1,\, X+Y\ge1,\, X<\frac1{1-s},\,Y\ge \frac1{1+s}\big\}.
$$ 
Since $G(s,0,Y)=0$, this is sufficient to obtain the required goal. 

We have
$$
\partial_X G(s,X,Y)=\frac{\sinh Y}{e^{sY}\sinh^2(X+Y)}\Big[e^{s(X+Y)}\big(\cosh(X+Y)-s\sinh(X+Y)\big) -1\Big].
$$
But the expression in brackets is easily seen to equal
$$
\frac12(1-s)\Big(e^{(1+s)(X+Y)}-1 +\frac{1+s}{1-s} \big(e^{-(1-s)(X+Y)}-1 \big) \Big),
$$
hence, taking into account that $\sinh^2(X+Y)\simeq e^{2(X+Y)}$ (for $X+Y\ge1$), it remains to show that
\begin{equation}\label{bond}
\frac{\sinh Y}{e^{sY}e^{2(X+Y)}}\Big(e^{(1+s)(X+Y)}-1 -\frac{1+s}{1-s} \big(1-e^{-(1-s)(X+Y)}\big) \Big)\simeq 1, \qquad (s,X,Y)\in \calR.
\end{equation}
We now rewrite the left-hand side of the above to the form
$$
\frac{\sinh Y \big(e^{(1+s)(X+Y)}-1\big)}{e^{sY}e^{2(X+Y)}}\times \Big[1 -\frac{1+s}{1-s}\frac{1-e^{-(1-s)(X+Y)}}{e^{(1+s)(X+Y)}-1} \Big]
$$
and claim, which is enough for \eqref{bond}, that 
\begin{equation}\label{bond2}
L:= \frac{\sinh Y \big(e^{(1+s)(X+Y)}-1\big)}{e^{sY}e^{2(X+Y)}}  \simeq 1, \qquad (s,X,Y)\in \calR,
\end{equation}
and  
\begin{equation}\label{bond3}
R:=1 -\frac{1+s}{1-s}\frac{1-e^{-(1-s)(X+Y)}}{e^{(1+s)(X+Y)}-1}\simeq 1, \qquad|s|<1,\,\, X+Y\ge 1. 
\end{equation}

For \eqref{bond2} we first remark that $(1+s)Y\ge 1$ implies $(1+s)(X+Y)> 1$ and so $e^{(1+s)(X+Y)}-1\simeq e^{(1+s)(X+Y)}$. Thus
$$
L\simeq  \frac{\sinh Y}{e^{sY}}e^{-(1-s)(X+Y)}=\frac{\sinh Y}{e^{sY}}e^{-(1-s)X}e^{-(1-s)Y}\simeq \frac{\sinh Y}{e^{sY}}e^{-(1-s)Y},
$$
the last step due to $(1-s)X<1$. Now, for $Y>1$, obviously $L\simeq1$.  For $Y\le 1$, due to $Y\ge\frac1{1+s}$, we have $s>0$ and $\frac12< Y<1$,
hence $\sinh Y\simeq1$, $e^{sY}\simeq1$, and $e^{-(1-s)Y}\simeq1$, and  $L\simeq1$ again follows.

For \eqref{bond3}, with the substitution $T:=(1-s)(X+Y)$, $U:=(1+s)(X+Y)$, so that $T+U\ge2$, and setting $f(T):=\frac{1-e^{-T}}T$, $g(U):=\frac{e^U-1}U$, 
the proof reduces to checking that for some $0<\ve<1$ it holds 
$$
\frac{f(T)}{g(U)}<\ve, \qquad T>0,\,U>0,\,\, T+U\ge2.
$$
But $f$ is decreasing on $(0,\infty)$ and  $g$ is increasing on $(0,\infty)$, hence for any $U\ge2$ and $T>0$ we have
$$
\frac{f(T)}{g(U)}<\frac{f(0^+)}{g(2)}=\frac2{e^2-1}=:\ve_1<1.
$$
For $0<U<2$ and $T\ge2-U$, since $f$ and $g$ are continuous on $[0,2]$ and $g>0$ on $[0,2]$, we have  
$$ 
\frac{f(T)}{g(U)}\le\frac{f(2-U)}{g(U)}=\sup_{[0,2]}\frac{f(2-\cdot)}{g(\cdot)}= :\ve_2<1.
$$
Checking the last inequality is routine. (For instance, $g(U)$ can be replaced by $\frac2{2-U}$ which is smaller and $f(2-U)$ can be replaced by
$1-\frac{1+e^{-2}}4(2-U)$ which is bigger.) Now we can take $\ve=\max(\ve_1,\ve_2)$.

\vskip 0.4cm
\noindent \textbf{Case  (C)}:  $(1-s)X\ge1,\,(1+s)Y<1$. As we already mentioned,  \eqref{lala2} (C) follows from \eqref{lala2} (B) (and vice versa) 
due to $G(s,Y,X)=G(-s,X,Y)$ and appropriate relations between the $\calR$-regions. 

\vskip 0.4cm
\noindent  \textbf{Case (D)}: $(1-s)X\ge1,\,(1+s)Y\ge1$. Our goal is
$$
1-\frac{\exp(sX)\sinh Y+ \exp(-sY)\sinh X}{\sinh(X+Y)}\simeq 1.
$$
The upper bound, by 1,  is obvious and for the lower bound we need  the estimate
\begin{equation*}
F_s(X,Y):=\frac{\exp(sX)\sinh Y+ \exp(-sY)\sinh X}{\sinh(X+Y)}<\varepsilon,
\end{equation*}
for some $0<\varepsilon<1$, uniformly in $$\calR:=\big\{(s,X,Y)\colon |s|<1, X+Y\ge1, X\ge\frac1{1-s}, Y\ge\frac1{1+s}\big\}.$$ In fact, we shall check that 
\begin{equation}\label{bound2}
F_s(X,Y)\le F_s\Big(\frac1{1-s},\frac1{1+s}\Big)<\frac1{\cosh1}, \qquad (s,X,Y)\in\calR.
\end{equation}
But $F_s(X,Y)=F_{-s}(Y,X)$ and the definition of $\calR$ allow us to restrict $s$ to $[0,1)$, and we do this. 

Now, for the first part of \eqref{bound2} we first claim that for $s$ and $Y$ given, $F_s(X,Y)$ is decreasing as a function of $X$. Indeed, a routine calculation shows that 
$$
\partial_XF_s(X,Y)=\frac{\sinh Y}{e^{sY}}\Big[1-e^{s(X+Y)}\big(\cosh(X+Y)-s\sinh(X+Y)\big)  \Big]\frac1{\sinh^2(X+Y)}.
$$
Since for $s\le1$ the function $f_s(T):=e^T(\cosh T-s\sinh T)$ is increasing on $[0,\infty)$, and for $0\le s<1$, 
$$
f\Big(\frac{2s}{1-s^2}\Big)=e^{\frac{2s}{1-s^2}}\big(\cosh \frac{2s}{1-s^2}-s\sinh \frac{2s}{1-s^2}\big)
=\frac{1-s}2e^{\frac{2}{1-s}}+\frac{1+s}2e^{\frac{2}{1+s}}>\frac{1-s}2e^{\frac{2}{1-s}}> 1,
$$
it follows that (notice that $X+Y\ge \frac{2s}{1-s^2}$ for $(s,X,Y)\in\calR$)
$$
1-e^{s(X+Y)}\big(\cosh(X+Y)-s\sinh(X+Y)<0.
$$
Thus we completed verification of our claim. Since $F_s(X,Y)=F_{-s}(Y,X)$, we also infer that given $s$ and $X$, $F_s(X,Y)$ is decreasing as a function of $Y$, and this finishes checking the first inequality in \eqref{bound2}. 

For  the second inequality in \eqref{bound2}, a short calculation shows that
$$
F_s\Big(\frac1{1-s},\frac1{1+s}\Big)=\frac1{e}\Big(1+2\frac{\sinh\frac{1}{1-s}\sinh\frac{1}{1+s}}{\sinh\big(\frac{1}{1-s} + \frac{1}{1+s} \big) }\Big).
$$
To finish it suffices to check that
$$
\frac {\sinh\big(\frac{1}{1-s} + \frac{1}{1+s}\big)} {\sinh\frac{1}{1-s}\sinh\frac{1}{1+s}}=\coth \frac{1}{1+s}+\coth \frac{1}{1-s}  \ge2\coth1
$$
for $s\in[0,1)$. But as an easy calculation shows, the latter expression  increases as a function on $[0,1)$ (the fact that $\frac{\sinh u}u$  
increases on $(0,\infty)$ helps), and so the minimal value $2\coth1$ is attained for $s=0$.

\section{Appendix} \label{sec:app}

\noindent (A). We first describe the harmonic profile of $C_+^{(m)}$, $m\ge1$. The case of general cones is discussed, for instance, in \cite{BS} (see also  
\cite[Section 4.4.2]{MS} for additional comments). Recall that in polar coordinates 
$$
C_+^{(m)}=
\begin{cases}
\{\rho e^{i\theta}\colon \rho>0,\,0<\theta<\frac{\pi}m\}, \qquad m\,\,\, {\rm even},\\
\{\rho e^{i\theta}\colon \rho>0,\,|\theta|<\frac{\pi}{2m}\}, \qquad m\,\,\, {\rm odd}, 
\end{cases}
$$
and the harmonic profile of $C_+^{(m)}$ (unique up to a positive multiplicative constant), i.e. harmonic and positive in $C^{(m)}_+$, continuous on the closure 
of $C^{(m)}_+$, and vanishing on the boundary $\partial C^{(m)}_+$, is
$$
h_m(x)=\rho^m
\begin{cases}
\sin m\theta, \qquad m\,\,\, {\rm even},\\
\cos m\theta, \qquad m\,\,\, {\rm odd}; 
\end{cases}
$$
here $x=\rho e^{i\theta}$. Using the formulas for $\sin m\theta$ and $\cos m\theta$ 
gives for $m$ even
$$
h_m(x)=\sum_{j=0}^{m/2-1}(-1)^j\binom{m}{2j+1}x_1^{m-2j-1}x_2^{2j+1},
$$
and for $m$ odd  
$$
h_m(x)=\sum_{j=0}^{(m-1)/2}(-1)^j\binom{m}{2j}x_1^{m-2j}x_2^{2j}.
$$

It is seen that $h_m$ is a polynomial  of degree $m$ of two variables, consisting of $\lfloor\frac{m+1}2\rfloor$ monomials, each of degree $m$.
Consequently, we have the homogeneity property 
$$
h_m(x_1,x_2)=x_1^m h_m\big(1,\frac{x_2}{x_1}\big)=x_2^m h_m\big(\frac{x_1}{x_2},1\big)
$$
for $x_1\neq0$ and $x_2\neq0$, respectively. From the construction it follows that as a function on $\R^2$, $h_m$ for $m$ even vanishes on the rays 
$\rho e^{i\frac{\pi}m j}$ emanating from the origin, $\rho\ge0$, $j=0,1,\ldots, 2m-1$, and only there; for $m$ odd analogous statement concerns  the 
rays $\rho e^{i\frac{\pi}m(j+1/2)}$. 

Hence, for $m$ even, the polynomial of one variable and degree $m-1$
$$
P_{m}(u):=h_m(u,1)=\sum_{j=0}^{m/2-1}(-1)^j\binom{m}{2j+1}u^{m-2j-1}
$$
possesses exactly $m-1$ simple roots
$$
0\,\,\,{\rm and}\,\,\,\pm u_{j,m}:=\pm 1/\tan\big(\frac{\pi}m j\big), \qquad j=1,2,\ldots,\frac m2-1.
$$
For instance, $u_{1,4}=1$. Similarly, for $m$ odd, the polynomial of one variable and degree $m$ 
$$
P_{m}(u):=h_m(u,1)=\sum_{j=0}^{\frac{m-1}2}(-1)^j\binom{m}{2j}u^{m-2j}
$$
possesses exactly $m$ simple roots
$$
0\,\,\,{\rm and}\,\,\,\pm u_{j,m}:=\pm 1/\tan\Big(\frac{\pi}m(j+\frac12)\Big), \qquad j=0,1,\ldots,\frac{m-3}2.
$$
For instance, $u_{0,3}=\sqrt3$. Therefore,  for $m$ even, 
$$
P_{m}(u)=m\,u\prod_{j=1}^{\frac m2-1}(u-u_{j,m})(u+u_{j,m}),
$$
and consequently,
\begin{equation}\label{ee}
h_m(x_1,x_2)=mx_1x_2\prod_{j=1}^{\frac m2-1}(x_1-u_{j,m}x_2)(x_1+u_{j,m}x_2);
\end{equation}
by convention, for $m=2$ the last two products and are understood as 1. Similarly, for $m$ odd, 
$$
P_{m}(u)=u\prod_{j=0}^{\frac{m-3}2}(u-u_{j,m})(u+u_{j,m}),
$$
(the convention analogous to this above applies here and below  for $m=1$) and consequently,
\begin{equation}\label{eee}
h_m(x_1,x_2)=x_1\prod_{j=0}^{\frac{m-3}2}(x_1-u_{j,m}x_2)(x_1+u_{j,m}x_2).
\end{equation}
Observe that the sequences
$$
\{u_{j,m}\}_{j=1}^{\frac{m-2}2}\quad {\rm and} \quad \{u_{j,m}\}_{j=0}^{\frac{m-3}2}
$$
for $m$ even and odd, respectively, are decreasing. This property, together with the fact that $\frac1{u_{1,m}}=\tan\frac{\pi}m$ and 
$\frac1{u_{0,m}}=\tan\frac{\pi}{2m}$ for $m$ even and odd, respectively, allows to confirm positivity of $h_m$ on $C^{(m)}_+$ directly from the formulas 
\eqref{ee} and \eqref{eee}. In addition, since the rays emanating from the origin, $x_2=0$ and $x_2=\tan\big(\frac{\pi}m\big) x_1$, or 
$x_2=\pm\tan\big(\frac{\pi}{2m}\big) x_1$, for $m$ even or $m$ odd, respectively, form the boundary of $C^{(m)}_+$, the fact that $h_m$ vanishes on 
$\partial C^{(m)}_+$ is also directly seen.

In Table 1 we collect the explicit formulas of $h_m$ for $1\le m\le 6$. 
{{
\begin{table*}[htbp]
\centering
\begin{tabular}{|p{0.03\textwidth}||p{0.8\textwidth}|}\hline
$m$             & $h_m(x_1,x_2)$\\ \hline\hline 
$1$ & $ x_1 $  \\ \hline
$2$  & $2x_1x_2$ \\ \hline
$3$  & $x_1(x_1-\sqrt3 x_2)(x_1+\sqrt3 x_2)$  \\ \hline 
$4$   & $4x_1x_2(x_1-x_2)(x_1+x_2)$ \\ \hline
$5$   & $x_1(x_1-\sqrt{5+2\sqrt5}x_2)(x_1+\sqrt{5+2\sqrt5}x_2)(x_1-\sqrt{5-2\sqrt5}x_2)(x_1+\sqrt{5-2\sqrt5}x_2)$ \\ \hline
$6$   & $6x_1x_2(x_1-\sqrt3 x_2)(x_1+\sqrt3 x_2)(x_1-\frac1{\sqrt3}x_2)(x_1+\frac1{\sqrt3}x_2)$ \\ \hline
\end{tabular}
\vskip.2cm
\caption{The polynomials $h_m$ for low values of $m$}
\label{tab:Table3}
\end{table*}
}}

Let $h_{j,m}$, $j=1,\dots,m$, be the polynomials of two variables and degree 1 entering the decompositions \eqref{ee} or \eqref{eee} (the order of these polynomials is not important). We state the following conjecture.

\textbf{Conjecture}. For $m\ge1$ we have uniformly in $x,\,y\in C_+^{(m)}$, and  $t>0$, 
\begin{equation}\label{hip}
p_t^{\sgn,C_+^{(m)}}(x,y)\simeq \prod_{j=1}^{m}\frac{h_{j,m}(x)h_{j,m}(y)}{h_{j,m}(x)h_{j,m}(y)+t}\,p_t^{(2)}(x-y).
\end{equation}

The conjecture is supported by our results, Proposition \ref{prop:1} taken with $d=2$, $k=1,2$, and $\eta=\textbf{1}$ (this corresponds to $m=1,2$) and
Theorems \ref{thm:1} and \ref{thm:2} (these correspond to $m=3, 4$).

Observe that for $m\ge4$ some factors of the product in \eqref{hip} are comparable and so \eqref{hip} can be rewritten in an equivalent form. 
Namely, for even $m\ge4$ and $j=2,\ldots,\frac{m-2}2$, we have 
\begin{equation}\label{hip2}
\frac{(x_1\pm u_{j,m}x_2)(y_1\pm u_{j,m}y_2)}{(x_1\pm u_{j,m}x_2)(y_1\pm u_{j,m}y_2)+t}\simeq \frac{x_1y_1}{x_1y_1+t},\qquad x,y\in C_+^{(m)},\quad t>0,
\end{equation}
and the same relation holds for $j=1$ but only with the `$+$' sign. Therefore, \eqref{hip} is for even $m\ge4$ equivalent to 
$$
p_t^{\sgn,C_+^{(m)}}(x,y)\simeq 
\Big(\frac{x_1y_1}{x_1y_1+t}\Big)^{m-2}\frac{x_2y_2}{x_2y_2+t}\frac{(x_1-u_{1,m}x_2)(y_1-u_{1,m}y_2)}{(x_1-u_{1,m}x_2)(y_1-u_{1,m}y_2)+t}\,p_t^{(2)}(x-y).
$$
Similarly, for odd $m\ge5$, \eqref{hip2} holds for $j=1,\ldots,\frac{m-3}2$. Therefore,  \eqref{hip}  is for odd $m\ge3$ equivalent to 
\begin{align*}
&p_t^{\sgn,C_+^{(m)}}(x,y)\simeq \\
&\Big(\frac{x_1y_1}{x_1y_1+t}\Big)^{m-2}\frac{(x_1-u_{0,m}x_2)(y_1-u_{0,m}y_2)}{(x_1-u_{0,m}x_2)(y_1-u_{0,m}y_2)+t}
\frac{(x_1+u_{0,m}x_2)(y_1+u_{0,m}y_2)}{(x_1+u_{0,m}x_2)(y_1+u_{0,m}y_2)+t}\,p_t^{(2)}(x-y).
\end{align*}

Finally we observe that the above two formulas can be unified into one. Let $L_{1,m}$ and $L_{2,m}$ be the boundary half-lines of $C_+^{(m)}$, i.e. 
$\{(s,0)\colon s\ge0\}$ and $\{(s,\frac1{u_{1,m}}s)\colon s\ge0\}$ for $m$ even, and $\{(s,\pm\frac1{u_{0,m}}s)\colon s\ge0\}$ for $m$ odd. 
Let  $\delta_{L_{j,m}}(x)$ and $\delta_{\rm v}(x)$ denote the distance of $x\in C_+^{(m)}$ to $L_{j,m}$, $j=1,2$, and to the vertex of $C_+^{(m)}$, respectively. In this new \textit{entourage}, applying $\frac{A}{A+t}\simeq 1\wedge \frac At$, $A,t>0$, shows that  for $m\ge2$  \eqref{hip} is equivalent to
\begin{equation}\label{hip22}
p_t^{\sgn,C_+^{(m)}}(x,y)\simeq\Big(1\wedge\frac{\delta_{\rm v}(x)\delta_{\rm v}(y)}{t}\Big)^{m-2}\Big(1\wedge \frac{\delta_{L_{1,m}}(x)\delta_{L_{1,m}}(y)}{t}\Big)
\Big(1\wedge \frac{\delta_{L_{2,m}}(x)\delta_{L_{2,m}}(y)}{t}\Big)\,p_t^{(2)}(x-y).
\end{equation} 

\vskip0.4cm
(B). As an application of \eqref{hip22} we now state a genuinely sharp estimate for the Dirichlet heat kernel on the domain which is an intersection of 
two half-spaces in $\R^d$, $d\ge3$, with the angle $\theta=\frac{\pi}m$, $m=3,4$, between the hyperplanes supporting them. This is in connection with 
\cite[Theorem 3.1]{Ser}, where upper bounds for the Dirichlet heat kernels on intersections of two half-spaces with obtuse angle between their supporting hyperplanes were obtained. Obviously, \eqref{hip22} is conjectured, nevertheless it is convenient to work with general $m$, although the conclusions 
are valid only for $m=2,3,4$ (the case of $m=2$ will follow by Proposition \ref{prop:1} taken with $d\ge3$ and $\eta=\textbf{1}$).

The domain we consider is a Weyl chamber. With identification $\R^d=\R^{d-2}\times\R^2$ and up to rotation, this is the product of $\R^{d-2}$ 
and $C_+^{(m)}\subset\R^2$. More precisely, this framework is related to the root system $R$ with orthogonal partition $R=R_1\sqcup R_{2,m}$, where 
$R_1=\{\pm e_1,\ldots,\pm e_{d-2},0,0\}$ and 
$$
R_{2,m}=\Big\{\Big(0,\ldots,0, \pm \cos\big(\frac{2\pi}m j\big), \pm \sin\big(\frac{2\pi}m j\big)\Big)\colon j=0,\ldots,2d-1 \Big\}.
$$ 
Then the corresponding Weyl chamber is $H_m=H_{1,m}^+\cap H_{2,m}^+$, where, for $k=1,2$, 
$$
H_{k,m}^+=\{x\in\R^d\colon \langle x,v_{k,m}\rangle>0\},
$$
and $v_{1,m}$, $v_{2,m}$, are the unit inside vectors orthogonal to the facets of $H_{1,m}^+$ and $H_{2,m}^+$, respectively. More precisely, 
for $m$ even, $v_{1,m}=(0,\ldots,0,0,1)$, $v_{2,m}=(0,\ldots,0,\sin(\pi/m),-\cos(\pi/m))$, and, for $m$ odd, 
$v_{1,m}=(0,\ldots,0,\sin(\pi/2m),\cos(\pi/2m))$, $v_{2,m}=(0,\ldots,0,\sin(\pi/2m),-\cos(\pi/2m))$. In both cases,
$$
\angle(H_{1,m}^+, H_{2,m}^+):=\pi-\angle(v_{1,m},v_{2,m})=\frac{\pi}{m}.
$$

The Dirichlet heat kernel for $H_m$ is (see, for instance, \cite[Proposition 2.4]{S2}, stated in a slightly more general setting)
$$
p_t^{H_m}(x,y)=p_t^{(d-2)}(x'-y')p_t^{C_+^{}(m)}(x'',y''),
$$
where $x=(x',x'')\in \R^{d-2}\times\R^2$, and similarly for $y\in\R^d$.

Now, let $S_{k,m}$, $k=1,2$, denote the two flat parts of the boundary of $H_m$, $S_{k,m}=\overline{H_m}\cap \langle v_{k,m}\rangle^\bot$, that correspond 
to $L_{k,m}$, in the sense that for $x=(x',x'')\in \partial H_m$ we have $x\in S_{k,m}$ if and only if $x''\in L_{k,m}$. This also means that  
$x=(x',x'')\in \partial H_m$ belongs to the edge $E_m:=S_{1,m}\cap S_{2,m}$ if and only if $x''$ is the vertex of $C_+^{(m)}$. Next, let 
$\delta_{S_{k,m}}(x)$ and $\delta_{\rm E}(x)$ denote the distance of $x\in H_m$ to $S_{k,m}$, $k=1,2$, and to the edge $E_m$, respectively. 
Observe that for $x=(x',x'')\in H_m$ we have $\delta_{S_{k,m}}(x)=\delta_{L_{k,m}}(x'')$ and $\delta_{\rm E}(x)=\delta_{\rm v}(x'')$. Thus, \eqref{hip22} 
transforms onto
\begin{equation}\label{hip222}
p_t^{H_m}(x,y)\simeq\Big(1\wedge\frac{\delta_{\rm E}(x)\delta_{\rm E}(y)}{t}\Big)^{m-2}\Big(1\wedge \frac{\delta_{S_{1,m}}(x)\delta_{S_{1,m}}(y)}{t}\Big)
\Big(1\wedge \frac{\delta_{S_{1,m}}(x)\delta_{S_{1,m}}(y)}{t}\Big)\,p_t^{(d)}(x-y),
\end{equation}
uniformly in $x,y\in H_m$ and $t>0$.

\vskip0.4cm
(C). The aim of this part is to show that the conjectured estimate \eqref{hip} is compatible with estimates proved in \cite{GSC} in 
the general framework of an arbitrary inner uniform domain, in the sense that \eqref{hip} is stronger. In particular, in the considered settings of 
$C^{(3)}_+$ and  $C^{(4)}_+$, the estimates \eqref{sgn} and \eqref{sgn1} are stronger than those that follow from \cite{GSC}. 

Rather than to compare \eqref{hip} with \eqref{convex2} taken for $U=C^{(m)}_+$ (so that $d_U(x,y)=\|x-y\|$), we prefer to use a more explicit variant 
of \eqref{convex2} specialized to \textit{special Lipschitz domains} (sLd for short, sometimes called simply \textit{Lipschitz domains}), which is a 
subclass of \textit{strongly Lipschitz domains}; this variant is taken from \cite[Corollary 6.14]{GSC}. We also prefer to furnish first a detailed 
argument in the specific case of $m=4$ and then to generalize it to an arbitrary $m$.

An sLd is, up to  rotation, the domain of the form 
$$
U_\Phi=\{x\in\Rd\colon \Phi(x')<x_d\}, \qquad x=(x',x_d),
$$ 
where $\Phi\colon \R^{d-1}\to\R$, $d\ge2$, is a Lipschitz function; plainly, $U_\Phi$ is a domain above the graph of $\Phi$. Clearly, any Weyl chamber 
in $\Rd$, $d\ge2$, being an open polyhedral cone with vertex at the origin, is an example of a convex sLd. 

Any sLd is a \textit{uniform domain} (see \cite[Proposition 6.6]{GSC}) and any convex sLd is an inner uniform domain. Thus, any Weyl chamber is an example 
of a convex unbounded uniform domain. Therefore, \cite[Corollary 6.14]{GSC} (repeated also as \cite[Corollary 6.28]{GSC}; note that in both places the 
exponent 2 is missing) adapted to the setting of $C^{(4)}_+$ says that 
\begin{equation}\label{xyz}
p_t^{\sgn,C^{(4)}_+}(x,y)\simeq \simeq t^{-d/2}\frac{h(x)h(y)}{h(x_{*,\sqrt t})h(y_{*,\sqrt t})}\, e^{-c\|x-y\|^2/t}, \qquad x,y\in C^{(4)}_+,
\end{equation}
where $z_{*,\sqrt t}=z+\sqrt t\, \vec{v}_2$ for $z\in  C^{(4)}_+$, and 
$$
\vec{v}_2=(\tau^+,\tau^-), \qquad  \tau^{\pm}=\frac12\sqrt{2\pm \sqrt2},
$$
is the unit vector generating the symmetry line of $C^{(4)}_+$, and 
$$
h(x):=\frac14 h_4(x)=x_1x_2(x_1-x_2)(x_1+x_2), \qquad x=(x_1,x_2)\in C^{(4)}_+,
$$
is the harmonic profile of $C^{(4)}_+$. 

Originally, in \cite[Corollary 6.14]{GSC}, the  `adjusted' point $z_{\sqrt t}$ was defined as $z_{\sqrt t}=z+\sqrt t\,\vec{e}_d$, and our modified 
version $z_{*,\sqrt t}$ reflects necessity of using a rotated system of coordinates in which $C^{(4)}_+$ indeed is a domain above the graph of a 
Lipschitz function. In this new system, the unit vectors $\vec v_1$ and $\vec v_2$, where $\vec v_1$ is such that $\vec v_1\bot\vec v_2$ and the basis 
$(\vec v_1,\vec v_2)$ is positively oriented, are coordinate vectors of orthogonal axes. 

Speaking more precisely, in the new coordinates $(x_1',x_2')$ we consider
$$
U_\Phi=\{x'\in\R^2\colon x_2'>(\sqrt 2+1)|x_1'|\},\qquad \Phi(x_1')=(\sqrt 2+1)|x_1'|,
$$
so that the clockwise rotation $R_0:=R(\theta_0)$ with $\theta_0=3\pi/8$ moves $\vec e_2=(0,1)$ onto $\vec v_2$ and maps $U_\Phi$ onto $C^{(4)}_+$. Recall 
that $\sin\theta_0=\tau^+$ and $\cos\theta_0=\tau^-$, so that $\tan\theta_0=\sqrt2+1$, and if $R_0(z')=z$, where $z'\in U_\Phi$, then 
$R_0(z_{\sqrt t}')=z_{*,\sqrt t}$.

To see that \eqref{xyz} is implied by \eqref{sgn} we first observe that the product of the numerators of the four fractions preceding $p_t^{(2)}(x-y)$ 
in \eqref{sgn} equals $h(x)h(y)$. Therefore, all we need are the following relations between the products of the denominators in \eqref{sgn} and \eqref{xyz} 
\begin{equation}\label{uuu}
h(x_{*,\sqrt t})h(y_{*,\sqrt t})\, e^{-\varepsilon\|x-y\|^2/t}\lesssim D(x,y,t)\lesssim h(x_{*,\sqrt t})h(y_{*,\sqrt t}), \qquad x,y\in C^{(4)}_+,\quad t>0,
\end{equation}
where $\varepsilon>0$ is arbitrarily fixed and 
$$
D(x,y,t)=(x_1y_1+t)(x_2y_2+t)\big((x_1-x_2)(y_1-y_2)+t\big)\big((x_1+x_2)(y_1+y_2)+t\big).
$$
With the aid of \eqref{uuu}, \eqref{xyz} follows from \eqref{sgn} with $c_1=\frac14$ and any $c_2<\frac14$.

To check \eqref{uuu} we first remark that by homogeneity it suffices to consider $t=1$ only. Now, recall that $x_{*,1}=(x_1+\tau^+,x_2+\tau^-)$ 
for $x=(x_1,x_2)$ and hence 
$$
h(x_{*,1})=(x_1+\tau^+)(x_2+\tau^-)\big((x_1-x_2)+(\tau^+-\tau^-)\big)\big((x_1+x_2)+(\tau^++\tau^-)\big),
$$
and analogously for $h(y_{*,1})$. Now, the right-hand side of \eqref{uuu}, with $t=1$, follows by applying the simple estimate
$$
ab+1\le \max\big(1,\frac1{\tau}\big)(a+\tau)(b+\tau), \qquad a,b>0,\quad \tau>0,
$$
to each of the four factors of $D(x,y,1)$ separately, with $\tau=\tau^+,\tau^-,\tau^+-\tau^-,\tau^+ +\tau^-$, respectively. For the left-hand side of 
\eqref{uuu} with $t=1$, for fixed $\varepsilon>0$ and  given $\tau>0$  we use \eqref{ort2} to obtain
$$
(a+\tau)(b+\tau)\lesssim (a+1)(b+1)\lesssim (ab+1)e^{\ve(a-b)^2}, \qquad a,b>0.
$$
As in the previous case we  pair appropriate two factors in $h(x_{*,1})h(y_{*,1})$, to apply \eqref{ort2} (in four steps).

We now pass to the general $m\ge3$, where in \eqref{xyz} $C^{(4)}_+$ is replaced by $C^{(m)}_+$, and $h$ by $h_m$. We first consider $m$ even. 
The unit vector generating  the symmetry line of $C^{(m)}_+$ is 
$$
\vec{v}_{2,m}=(\cos \frac{\pi}{2m},\sin \frac{\pi}{2m} )=:(\tau^+_m,\tau^-_m), 
$$
and the harmonic profile of  $C^{(m)}_+$ is given by \eqref{ee}. It is important to observe that
$$
u_{1,m}=1/\tan\frac{\pi}m< \frac{\tau^+_m}{\tau^-_m}=1/\tan\frac{\pi}{2m},
$$
hence, since $\{u_{j,m}\}_{j=1}^{\frac m2 -1}$ is decreasing, we also have $\tau^+_m-u_{j,m}\tau^-_m>0$ for $j=1,\ldots,\frac m2-1$. 

It is obvious that what remains to prove in this general case is the  variant of \eqref{uuu} for $t=1$, with $h$ replaced by $h_m$, and $D(x,y,1)$ replaced by
\begin{align*}
D_m(x,y,1)&=(x_1y_1+1)(x_2y_2+1)\times \\
&\,\,\,\,\prod_{j=1}^{\frac m2-1}\big((x_1-u_{j,m}x_2)(y_1- u_{j,m}y_2)+1\big)\big((x_1+u_{j,m}x_2)(y_1+u_{j,m}y_2)+1\big).
\end{align*}
But this goes with exactly the same arguments as in the specific case $m=4$ once we note that for $x\in C^{(m)}_+$ we have $x_{*,1}=(x_1+\tau^+_m,x_2+\tau^-_m)$ and
$$
h_m(x_{*,1})=(x_1+\tau^+_m)(x_2+\tau^-_m)
\prod_{j=1}^{\frac m2-1}\big((x_1-u_{j,m}x_2)+(\tau^+_m -u_{j,m}\tau^-_m)\big)\big((x_1+x_2u_{j,m})+(\tau^+_m +u_{j,m}\tau^-_m)\big),
$$
and analogously for $h(y_{*,1})$. 

The case of $m$ odd is conceptually simpler since then $C^{(m)}_+$ is an sLd from the perspective of the $x_2$-axis (no rotation is necessary) and 
$\vec{v}_{2,m}=(1,0)$ so that $x_{*,1}=(x_1+1,x_2)$ for $x=(x_1,x_2)$. Repeating previous reasoning, with necessary adjustments, brings the expected conclusion.  

\vskip0.4cm
\noindent (D) We now shortly comment on the issue of continuity of $P(s,X,Y)$ defined in Section \ref{ssec:second}.  It suffices to check that for any $r>1$, 
$|P_m(s,X,Y)|\le p_m(r)$ on $(-r,r)^3$, with $\sum_3^\infty p_m(r)<\infty$. Such estimate is easily seen by using crude estimates of 
$R_{m,j}(X,Y)$ and $Q_{m,j}(s)$ from the restricted range of $s,X,Y$. 

Indeed, for $m\ge3$ odd, we have $Q_{m,j}(s)\le mr^{m-3}$ and $|R_{m,j}(X,Y)|\le m r^{m-3}$ for $|s|<r$, $|X|<r$, $|Y|<r$, and $j=0,\ldots,\frac{m-3}2$. 
Hence, using \eqref{pmo}, for $(s,X,Y)\in (-r,r)^3$ we have
$$ 
|P_m(s,X,Y)|\le m^2r^{2(m-3)}\sum_{j=0}^{\frac{m-3}2}\frac{1}{\Gamma(2j)\Gamma(m-2j-2)}.
$$
But $\max(2j,m-2j-2)\ge\frac{m-1}2$ for $0\le j\le \frac{m-3}2$, hence the latter sum is bounded by $\frac{m-1}2/\Gamma(\frac{m-1}2)$ and finally, for $m$ odd,
$$
|P_m(s,X,Y)|\le  m^3r^{2(m-3)}/\Gamma\big(\frac{m-1}2\big):=p_m(r),
$$
with $\sum_{m\,\,{\rm odd}}p_m(r)<\infty$. For $m$ even, using \eqref{pme}  gives the analogous conclusion.

\end{document}